\documentclass[12pt,a4paper]{article}
\usepackage{amsmath}
\usepackage{amssymb}
\usepackage{accents}

\usepackage{mathtools}
\usepackage{amsthm}
\usepackage[usenames]{color}

\usepackage{graphicx}

\newcommand{\1}[1]{{\mathbf 1}{\{#1\}}}

\newcommand{\vr}{\varrho}

\newcommand{\eps}{\varepsilon}
\newcommand{\Z}{{\mathbb Z}}

\newcommand{\B}{{\mathsf B}}

\newcommand{\V}{{\mathcal V}}

\newcommand{\LL}{{\mathcal L}}
\newcommand{\G}{{\mathcal G}}
\newcommand{\D}{{\mathcal D}}

\newcommand{\R}{{\mathbb R}}

\newcommand{\EE}{{\mathcal E}}
\newcommand{\F}{{\mathcal F}}

\newcommand{\RI}{\mathop{\mathrm{RI}}}

\newcommand{\s}{{\widehat S}}

\newcommand{\JJ}{{\mathfrak J}}

\let\phi=\varphi

\newcommand{\E}{{\mathbb E}}

\newcommand{\hpsi}{{\widehat\psi}}
\newcommand{\hphi}{{\widehat\varphi}}

\newcommand{\tW}{{\widetilde W}}

\newcommand{\diam}{{\mathop{\mathrm{diam}}}}
\newcommand{\8}{{\infty}}

\newcommand{\he}{\widehat{e}}
\newcommand{\hL}{\widehat{L}}

\newcommand{\eqlaw}{\stackrel{\text{\tiny law}}{=}}

\newcommand{\convlaw}{\stackrel{\text{\tiny law}}{\longrightarrow}}

\newcommand{\IP}{{\mathbb P}}
\newcommand{\IE}{{\mathbb E}}

\newcommand{\tP}{\widetilde{P}}

\DeclareMathSymbol{\widehatsym}{\mathord}{largesymbols}{"62}

\newcommand{\hZ}{\widehat{Z}}
\newcommand{\tZ}{\widetilde{Z}}

\newcommand{\capa}{\mathop{\mathrm{cap}}}

\newcommand{\hm}{\mathop{\mathrm{hm}}\nolimits}
\newcommand{\hhm}{\mathop{\widehat{\mathrm{hm}}}\nolimits}
\newcommand{\Es}{\mathop{\mathrm{Es}}\nolimits}

\newcommand{\dist}{\mathop{\mathrm{dist}}}

\newcommand{\htau}{\widehat{\tau}}

\newcommand{\tS}{{\widetilde S}}
\newcommand{\tN}{{\widetilde N}}
\newcommand{\tL}{{\widetilde\LL}}
\newcommand{\ttL}{{\widetilde L}}
\newcommand{\tH}{{\widetilde H}}

\newcommand{\ttau}{{\tilde\tau}}

\allowdisplaybreaks

\newtheorem{theo}{Theorem}[section]
\newtheorem{lem}[theo]{Lemma}
\newtheorem{df}[theo]{Definition}
\newtheorem{prop}[theo]{Proposition}

\title{The vacant set of two-dimensional critical
 random interlacement is infinite}
\author{Francis Comets$^{1}$ \and
 Serguei~Popov$^{2}$}

\begin{document}

\maketitle

{\footnotesize 
\noindent $^{~1}$Universit\'e Paris Diderot -- Paris 7, 
Math\'ematiques, 
 case 7012, F--75205 Paris
Cedex 13, France
\\
\noindent e-mail:
\texttt{comets@math.univ-paris-diderot.fr}

\noindent $^{~2}$Department of Statistics, Institute of Mathematics,
 Statistics and Scientific Computation, University of Campinas --
UNICAMP, rua S\'ergio Buarque de Holanda 651,
13083--859, Campinas SP, Brazil\\
\noindent e-mail: \texttt{popov@ime.unicamp.br}

}

\begin{abstract}
For the model of two-dimensional 
 random interlacements in the critical regime (i.e., $\alpha=1$),
we prove that the vacant set is a.s.\ infinite, thus solving an
open problem from~\cite{CPV15}. Also, we prove that the entrance
measure of simple random walk on annular domains has certain 
regularity properties; this result is useful when dealing
with soft local times for excursion processes.
 \\[.3cm]\textbf{Keywords:} random interlacements, vacant set, 
 critical regime, simple random walk, Doob's $h$-transform, annular domain
\\[.3cm]\textbf{AMS 2010 subject classifications:}
Primary 60K35. Secondary 60G50, 82C41.
\end{abstract}

\section{Introduction and results}
\label{s_intro}
The model of random interlacements,
recently introduced by Sznitman~\cite{Szn10}, 
has proved its usefulness for studying
fine properties of traces left by simple random walks on graphs.
The ``classical'' random interlacements is a Poissonian soup
of (transient) simple random walks' trajectories in~$\Z^d$,
$d\geq 3$; we refer to recent books~\cite{CT12, DRS14}.
Then, the model of two-dimensional random interlacements was
introduced in~\cite{CPV15}.
Observe that, in two dimensions, even a single trajectory
of a simple random walk is space-filling.
Therefore, to define the process in a meaningful way,
one uses the SRW's trajectories \emph{conditioned}
on never hitting the origin, see the details below.
We observe 
also that the use of conditioned trajectories 
to build the interlacements goes back to Sznitman~\cite{Szn14},
see the definition of 
``tilted random interlacements'' there.
%  on page~3 
% \textbf{(check the page, when published)}. 
Then, it is known (Theorem~2.6 of~\cite{CPV15}) that, 
for a random walk on a large torus conditioned on not 
hitting the origin up to some time proportional to the mean cover time, 
the law of the vacant set around the origin is close
to that of random interlacements at the corresponding level.
This means that, similarly to higher-dimensional case,
 two-dimensional random interlacements
have strong connections to random walks on discrete tori.

% \subsection{Formal definitions and results}
% \label{s_defs_results}
Now, let us recall the formal construction of
 (discrete) two-dimensional random interlacements.

In the following, $\|\cdot\|$ denotes the Euclidean norm
in~$\R^2$ or $\Z^2$, and $\B(x,r)=\{y: \|x-y\|\leq r\}$ is the 
(closed) ball of radius~$r$ centered in~$x$. 
We write $f(n)=O(g(n))$ as $n\to\infty$
when, for some constant~$M>0$, $|f(n)|\leq M|g(n)|$ for all
large enough~$n$; $f(n)=o(g(n))$ means that
 $\lim_{n\to \infty}\frac{f(n)}{g(n)}=0$.
Also, we write $f(n)\sim g(n)$ when $\lim_{n\to \infty}\frac{f(n)}{g(n)}=1$. In fact, the same symbol is also used 
for neighboring sites: for $x,y\in\Z^d$ we
write $x\sim y$ if $\|x-y\|=1$, that is, $x$ and~$y$ 
are neighbors. Hopefully, this creates no confusion
since the meaning of ``$\sim$''
 will always be clear from the context.

% \subsection{Two-dimensional random interlacements: the discrete case}
% \label{s_int_discr}
Let~$(S_n, n\geq 0)$ be two-dimensional simple
random walk. Write~$\IP_x$ for the law of the walk started from~$x$
and~$\IE_x$ for the corresponding expectation.
Let
\begin{align}
\tau_0(A) &= \inf\{k\geq 0: S_k\in A\} \label{entrance_t},\\
\tau_1(A) &= \inf\{k\geq 1: S_k\in A\} \label{hitting_t}
\end{align}
be the entrance and the hitting time of the set~$A$ by 
simple random walk~$S$ (we use the convention $\inf \emptyset = +\8$). 
Define the potential kernel~$a$ by
\begin{equation}
\label{def_a(x)}
a(x) = \sum_{k=0}^\infty\big(\IP_0[S_k\!=\!0]-\IP_x[S_k\!=\!0]\big).
\end{equation}
It can be shown that the above series indeed converges 
and we have~$a(0)=0$, $a(x)>0$ for $x\neq 0$, and
\begin{equation}
\label{formula_for_a}
 a(x) = \frac{2}{\pi}\ln \|x\| + \frac{2\gamma+\ln 8}{\pi} 
 + O(\|x\|^{-2}) 
\end{equation}
as $x\to\infty$, where $\gamma=0.5772156\dots$ 
is the Euler-Mascheroni constant,
cf.\ Theorem~4.4.4 of~\cite{LL10}.
% (the value of $\gamma'$ is 
% known\footnote{$\gamma'=\pi^{-1}(2\gamma+\ln 8)$, 
% where $\gamma=0.5772156\dots$ is the Euler-Mascheroni constant}, 
% but we will not need it in
% this paper).
Also, the function~$a$ is harmonic outside the origin, i.e.,
\begin{equation}
\label{a_harm}
 \frac{1}{4}\sum_{y: y\sim x}a(y) = a(x) \quad \text{ for all }
 x\neq 0.
\end{equation}
Observe that~\eqref{a_harm} implies 
that $a(S_{k\wedge \tau_0(0)})$ is a martingale.

The \emph{harmonic measure} of a finite $A\subset\Z^2$
is the entrance law ``starting at infinity''\footnote{observe
that the harmonic
measure can be defined in almost the same way in higher dimensions,
one only has to condition that~$A$ is eventually hit,
cf.\ Proposition~6.5.4 of~\cite{LL10}},
\begin{equation}
\label{def_hm}
 \hm_A(x) = \lim_{\|y\|\to\infty}\IP_y[S_{\tau_1(A)}=x]
\end{equation}
(see e.g.\ Proposition~6.6.1 of~\cite{LL10} for the 
proof of the existence of the above limit).
For a finite set~$A$ containing the origin,
we define its capacity by
\begin{equation}
\label{df_cap2}
\capa(A) = \sum_{x\in A}a(x)\hm_A(x);
\end{equation}
in particular, $\capa\big(\{0\}\big)=0$ since $a(0)=0$.
For a set not containing the origin, its capacity is defined
as the capacity of a translate of this set that does contain
the origin. Indeed, it can be shown that the capacity does not depend
on the choice of the translation.
Some alternative definitions are available,
cf.\ Section~6.6 of~\cite{LL10}.

Next, we define another random walk $(\s_n, n\geq 0)$
on~$\Z^2\setminus \{0\}$ in the following way:
the transition probability from~$x\neq 0$ to~$y$ 
equals $\frac{a(y)}{4a(x)}$ for all $x\sim y$.
Note that~\eqref{a_harm} implies that the random walk~$\s$
is indeed well defined, and, clearly, it is 
an irreducible Markov chain on~$\Z^2\setminus \{0\}$.
It can be easily checked that it is reversible
with the reversible measure~$a^2(\cdot)$, and transient
(for a quick proof of transience, just verify that
$1/a(\s)$ is a martingale outside the origin and its four 
neighbors, and use e.g.\ Theorem~2.5.8 of~\cite{MPW}).

For a finite~$A\subset \Z^2$,
define the \emph{equilibrium measure} with respect to the walk~$\s$:
\[
 \he_A(x) = \1{x\in A} 
 \IP_x\big[\s_k\notin A \text{ for all }k\geq 1\big] a^2(x),
\]
and the harmonic measure (again, with respect to the walk~$\s$)
\[
 \hhm_A(x) =  \he_A(x) \Big(\sum_{y\in A} \he_A(y)\Big)^{-1}.
\]
Also, note that~(13) and~(15) of~\cite{CPV15}
imply that $\he_A(x)=a(x)\hm_A(x)$ in the case $0\in A$, 
that is, the harmonic measure
for~$\s$ is the usual harmonic measure \emph{biased} by~$a(\cdot)$.
Now, we use the general construction of random interlacements
on a transient weighted graph introduced in~\cite{T09}.
In the following few lines we briefly summarize this
construction.
Let~$\mathcal{W}$ be the space of all doubly infinite
nearest-neighbour 
transient trajectories in~$\Z^2$,
\begin{align*}
 \mathcal{W} =& \big\{\vr=(\vr_k)_{ k\in \Z}: 
\vr_k\sim \vr_{k+1} \text{ for all }k;\\
&~~~~~~~~~~\text{ the set }
 \{m: \vr_m=y\} \text{ is finite for all }y\in\Z^2 \big\}.
\end{align*}
We say that~$\vr$ and~$\vr'$ are equivalent if they 
coincide after a time shift, i.e., $\vr\sim\vr'$
when there exists~$k$ such that $\vr_{m+k}=\vr_m'$ for all~$m$.
Then, let $\mathcal{W}^*=\mathcal{W}/\sim$ be the space 
of trajectories
modulo time shift, and define~$\chi^*$ to be the canonical
projection from~$\mathcal{W}$ to~$\mathcal{W}^*$. 
For a finite $A\subset \Z^2$, 
let~$\mathcal{W}_A$ be the set of trajectories in~$\mathcal{W}$ that intersect~$A$,
and we write~$\mathcal{W}^*_A$ for the image of~$\mathcal{W}_A$ under~$\chi^*$.
One then constructs the random interlacements as Poisson
point process on $\mathcal{W}^*\times \R^+$ with the intensity measure
$\nu\otimes du$, where~$\nu$ is described in the following
way. It is the unique sigma-finite measure on 
the cylindrical sigma-field  of~$\mathcal{W}^*$
such that for every finite~$A$
\[
 \mathbf{1}_{\mathcal{W}^*_A} \cdot \nu = \chi^* \circ Q_A,
\]
where the finite measure~$Q_A$ on~$\mathcal{W}_A$ is determined by the
following equality:
\[
Q_A\big[(\vr_k)_{k\geq 1}\in F, \vr_0=x, 
(\vr_{-k})_{k\geq 1}\in G\big]
=  \he_A(x) \IP_x[\s\in F] 
 \IP_x\big[\s\in G\mid \htau_1(A)=\infty\big].
\]
The existence and uniqueness of~$\nu$ was shown in 
Theorem~2.1 of~\cite{T09}.

\begin{df} \label{def:ri}
For a configuration $\sum_{\lambda}\delta_{(w^*_\lambda,u_\lambda)}$
of the above Poisson process, the process of 
two-dimensional random interlacements
at level~$\alpha$ (which will be referred to as RI($\alpha$))
is defined as the set of trajectories with label less than or equal 
to~$\pi\alpha$, i.e.,
\[
 \sum_{\lambda: u_\lambda\leq \pi\alpha}  
 \delta_{w^*_\lambda} .
\]
\end{df}
As mentioned in~\cite{CPV15}, in the above definition
it is convenient to pick the points with the $u$-coordinate
at most~$\pi\alpha$ (instead of just~$\alpha$, as in
the ``classical'' random interlacements model), since the 
formulas become generally cleaner.

It can be shown (see Section~2.1 of~\cite{CPV15},
in particular, 
Proposition~2.2 there) that the law of the vacant set~$\V^\alpha$
(i.e., the set of all sites not touched by the trajectories) of the
two-dimensional random interlacements can be uniquely
characterized by the following equality:
\begin{equation}
\label{eq_vacant2}
 \IP[A\subset \V^\alpha] = \exp\big(-\pi\alpha \capa(A)\big),
 \quad \text{ for all $A\subset\Z^2$ such that $0\in A$}.
\end{equation}
 
It is important to have in mind the following ``constructive''
description of the trace of RI($\alpha$) on
 a finite set $A\subset \Z^2$ such that $0\in A$. Namely,
\begin{itemize}
 \item take a Poisson($\pi\alpha\capa(A)$) number of particles;
 \item place these particles on the boundary of~$A$
 independently, with distribution
% $\overline{e}_A = \big((\capa A)^{-1}\he_A(x), x\in A\big)$;
$\hhm_A$;
 \item let the particles perform independent $\s$-random walks
 (since~$\s$ is transient, each walk only leaves a finite trace
 on~$A$).
\end{itemize}
In particular, note that~\eqref{eq_vacant2} is a direct
consequence of this description.

Some other basic properties of two-dimensional random
interlacements are contained in Theorems~2.3 and~2.5 of~\cite{CPV15}.
In particular, the following facts are known:
\begin{enumerate}
 \item The conditional translation invariance: 
 for all  $\alpha>0$, $x\in\Z^2 \setminus \{0\}$,
 $A\subset \Z^2$, and any
 lattice isometry~$\mathfrak{N}$ exchanging~$0$ and~$x$, we have 
\begin{equation}
\label{properties_RI_i'}
 \IP[A\subset\V^\alpha \mid x\in \V^\alpha]=
  \IP[\mathfrak{N}(A) \subset\V^\alpha \mid x\in \V^\alpha].
\end{equation}
 \item The probability that a given site is vacant
 is
\begin{equation}
 \label{properties_RI_ii}
\IP[x\in \V^\alpha]=\exp\Big(-\pi\alpha \frac{a(x)}{2}
\Big)
={\hat c}\|x\|^{-\alpha}\big(1+O(\|x\|^{-2})\big)
\end{equation}
(also, note that~\eqref{formula_for_a} yields
 an explicit expression for the constant ${\hat c}$
in~\eqref{properties_RI_ii}).
\item 
Clearly, \eqref{properties_RI_ii} implies that, as $r\to \infty$,
\begin{equation}
\label{expected_size_RI}
 \E \big(\vert \V^\alpha\cap \B(r)\vert \big) \sim
  \begin{cases}
   \text{const} \times r^{2-\alpha}, & \text{ for }\alpha < 2,\\
  \text{const} \times \ln r, & \text{ for }\alpha = 2,\\
   \text{const} , & \text{ for }\alpha > 2.
 \end{cases}
\end{equation}
\item For~$A$ such that $0\in A$
it holds that
\begin{equation}
\label{properties_RI_iii}
 \lim_{x\to \infty}\IP[A\subset\V^\alpha \mid x\in \V^\alpha]= 
   \exp\Big(-\frac{\pi\alpha}{4}\capa(A)
\Big).
\end{equation}
Informally speaking, if we condition that a very distant site is
vacant, this decreases the level of the interlacements 
around the origin by factor~$4$. A brief heuristic
explanation of this fact is given after (35)--(36) of~\cite{CPV15}.
\item The relation~\eqref{expected_size_RI} means that there is a
phase transition for the expected size of the vacant set at~$\alpha=2$.
However, the phase transition for the size itself occurs at~$\alpha=1$.
Namely, for $\alpha>1$ it holds that~$\V^\alpha$ is
 finite a.s., and for $\alpha \in (0,1)$
 we have $|\V^\alpha|=\infty$ a.s.
\end{enumerate}

The main contribution of this paper
is the following result: the vacant
set is a.s.\ infinite in the critical case $\alpha=1$:
\begin{theo}
\label{t_critical}
 It holds that $|\V^1|=\infty$ a.s.
\end{theo}
The above result may seem somewhat surprising, for the following
reason. As shown in~\cite{CPV15},
  the case $\alpha=1$ corresponds to the leading
term in the expression for the cover time of the two-dimensional
torus. It is known (cf.\ \cite{BK14,D12}), however, 
that the cover time has a
\emph{negative} second-order correction, 
which could be an evidence in favor of finiteness of~$\V^1$
(informally, the ``real'' all-covering regime should be
``just below'' $\alpha=1$).
On the other hand, it turns out that local fluctuations
of excursion counts overcome that negative correction,
thus leading to the above result. 
 
For $A\subset \Z^d$, denote by $\partial A=\{x\in A:\text{there
exists }y\notin A \text{ such that }x\sim y\}$ its internal
boundary.
Next, for simple random walk and a finite set $A\subset \Z^d$, 
let~$H_A$
be the corresponding Poisson kernel: for $x\in A$, $y\in\partial A$,
\begin{equation}
\label{df_Poisson_kernel}
 H_A(x,y) = \IP_x[S_{\tau_0(\partial A)}=y]
\end{equation}
(that is, $H_A(x,\cdot)$ is the exit measure from~$A$ starting at~$x$).
We need the following result, which states that,
if normalized by the harmonic measure, the entrance
measure to a large discrete ball is ``sufficiently regular''.
This fact will be an important tool for estimating large
deviation probabilities for soft local times \emph{without}
using union bounds with respect to sites of~$\partial A$
 (surely, the reader understands
that sometimes union bounds are just too rough). Also, 
we formulate it in all dimensions $d\geq 2$ for future 
reference\footnote{this fact is also needed at least 
in the paper~\cite{BGP}}.
\begin{prop}
\label{p_SRW_Hoelder}
Let $c>1$ and $\eps\in (0,1)$ be constants 
such that $c(1-\eps)>1+2\eps$, 
and abbreviate $A_n=(\B(cn)\setminus\B(n))\cup \partial\B(n)$.
Then, there exist positive constants $\beta,C$ (depending on~$c,\eps$,
and the dimension) such that for any 
$x\in \B(c(1-\eps)n)\setminus\B((1+2\eps)n)$
and any $y,z\in \partial\B(n)$ it holds that
\begin{equation}
\label{eq_SRW_Hoelder}
\Big|\frac{H_{A_n}(x,y)}{\hm_{\B(n)}(y)}
 - \frac{H_{A_n}(x,z)}{\hm_{\B(n)}(z)} \Big|  
\leq  C \Big(\frac{\|y-z\|}{n}\Big)^\beta
\end{equation}
for all large enough~$n$.
\end{prop}
We conjecture that the above should be true
with $\beta=1$, since one can directly
check that it is indeed the case for the Brownian motion
(observe that the harmonic measure on the sphere is uniform
in the continuous case and
see in Chapter~10 of~\cite{HFT} the formulas
for the Poisson kernel of the Brownian motion); however,
it is unclear to us how to prove that. In any case, 
\eqref{eq_SRW_Hoelder} is enough for our needs.

Next, we collect several technical facts we need in 
Section~\ref{s_toolbox}, and then prove our main
results in Section~\ref{s_proofs}. Also, in the end
of the paper we include a brief summary 
of notations, for reader's convenience.

\section{The toolbox}
\label{s_toolbox}
We collect here some facts needed for the proof of our main results. 
These facts are either directly available in the literature,
or can be rapidly deduced from known results.
Unless otherwise stated, we work in~$\Z^d$, $d\geq 2$.

We need first to recall some basic definitions
related to simple random walks in higher dimensions.
For $d\geq 3$ let~$G(x,y)=\IE_x\sum_{k=0}^\infty\1{S_k=y}$ 
denote the Green's function
(i.e., the mean number of visits to~$y$ starting from~$x$),
and abbreviate $G(y):=G(0,y)$.
For a finite set $A\subset\Z^d$ and $x,y\in A\setminus\partial A$
 define
\[
 G_A(x,y) = \IE_x \sum_{k=0}^{\tau_1(\partial A)-1} \1{S_k=y}
\]
to be the mean number of visits to~$y$ starting from~$x$
before hitting~$\partial A$
(since~$A$ is finite, this definition makes sense for all dimensions).
For $x\in A$ denote the \emph{escape probability} from~$A$
by $\Es_A(x)=\IP_x[\tau_1(A)=\infty]$.
The capacity of a finite set $A\subset\Z^d$
is defined by
\[
 \capa(A) = \sum_{x\in A} \Es_A(x).
\]
As for the capacity of a $d$-dimensional ball,
observe that Proposition~6.5.2 of~\cite{LL10} implies 
(recall that $d\geq 3$)
% for some (explicit)
% constants $C_d$, $d\geq 3$, we have
\begin{equation}
\label{capacity_ball_3}
 \capa(\B(n)) = \frac{(d-2)\pi^{d/2}}{\Gamma(d/2)d} 
 n^{d-2} + O(n^{d-3}).
\end{equation}
We also define the harmonic measure on~$A$ by $\hm_A(\cdot) 
= \frac{\Es_A(\cdot)}{\capa(A)}$.

Next, in Section~\ref{s_rw_annuli} we first collect some results
for simple random walks on annuli, namely: 
inside/outside exit probabilities 
(Lemmas~\ref{l_exit_balls} and~\ref{l_annulus_escape}), 
estimates on Green's functions restricted to
an annulus (Lemma~\ref{l_G_annulus}) and on exit
measures (Lemma~\ref{l_G_entrance}); also, we study 
escape probabilities from the inner boundary of an annulus
to the outer one in Lemma~\ref{l_escape_to}.
Then, we collect some facts related to the conditioned
walk~$\s$: an expression for probability of not hitting
a large ball, distant from the origin (Lemma~\ref{l_escape_from_ball}),
a formula for the (transient) capacity of such a ball
(Lemma~\ref{l_cap_distantball}), and a result that states
that the walks~$S$ and~$\s$ are almost indistinguishable
on ``distant'' sets (Lemma~\ref{l_relation_S_hatS}).
In Section~\ref{s_SLT} we first review the method 
of soft local times that permits us to construct
sequences of excursions of simple random walks and
random interlacements, and then
prove a result on large deviation for soft local times
(Lemma~\ref{l_LD_SLT}),
 using some machinery from the theory of empirical processes.
Next, in Lemma~\ref{l_consistency_SLT} we state 
another fact related to soft local times,
and, finally, we recall a result (Lemma~\ref{l_number_exc_torus})
 that permits us to control the number of excursions 
of simple random walk on torus.

\subsection{Basic estimates for the random walk on the annulus}
\label{s_rw_annuli}
Here, we formulate several basic facts about simple random walks 
on annuli.

\begin{lem}
\label{l_exit_balls}
\begin{itemize}
\item[(i)] For all $x \in \Z^2$ and $R>r>0$ 
such that
 $x \in \B(R)\setminus \B(r)$  we have
\begin{equation}
 \label{nothit_r_dim2}
\IP_x\big[\tau_1(\partial \B(R))< \tau_1(\B(r))\big] = 
\frac{\ln\|x\|-\ln r+O(r^{-1})}{\ln R-\ln r},
\end{equation}
as $r,R\to \infty$. 
\item[(ii)] For all $x \in \Z^d$, $d\geq 3$, and $R>r>0$ 
such that
 $x \in \B(R)\setminus \B(r)$  we have
\begin{equation}
 \label{nothit_r_dim3}
\IP_x\big[\tau_1(\partial \B(R))<\tau_1(\B(r))\big] = 
\frac{r^{-(d-2)}- \|x\|^{-(d-2)}+O(r^{-(d-1)})}
{r^{-(d-2)} - R^{-(d-2)}},
\end{equation}
as $r,R\to \infty$. 
\end{itemize}
\end{lem}
\begin{proof}
Essentially, this comes out of an application of the 
Optional Stopping Theorem to the martingales $a(S_{n\wedge \tau_0(0)})$
(in two dimensions) or $G(S_{n\wedge \tau_0(0)})$
(in higher dimensions).
See Lemma~3.1 of~\cite{CPV15} for the part~(i). As for the part~(2), 
apply the same kind of argument and use the expression
for the Green's function e.g.\ from Theorem~4.3.1 of~\cite{LL10}.
\end{proof}

\begin{lem}
\label{l_annulus_escape}
Let~$d\geq 2$ and let $c>1$ be fixed. Then for all large enough~$n$
we have for all $v\in (\B(cn)\setminus \B(n))\cup \partial\B(n)$
\begin{equation}
\label{annulus_escape}
 c_1\frac{\|v\|-n+1}{n} \leq 
\IP_v\big[\tau_1(\partial\B(cn))<\tau_1(\B(n))\big]
    \leq c_2\frac{\|v\|-n+1}{n}.
\end{equation}
with $c_{1,2}$ depending on~$c$.
\end{lem}
\begin{proof}
This follows from Lemma~\ref{l_exit_balls} together
with the observation that \eqref{nothit_r_dim2}--\eqref{nothit_r_dim3}
start working when $\|x\|-n$ become larger than a constant
(and, if~$x$ is too close to~$\B(n)$, we just pay a constant
price to force the walk out).
See also Lemma~8.5 of~\cite{SLT} (for $d\geq 3$) 
and Lemma~6.3.4 together with Proposition~6.4.1 
of~\cite{LL10} (for $d=2$).
\end{proof}

\begin{lem}
\label{l_G_annulus}
Let~$d\geq 2$.
Fix $c>1$ and $\delta>0$ such that $1+\delta<c-\delta$,
and abbreviate $A_n=(\B(cn)\setminus\B(n))\cup \partial\B(n)$. 
Then, there exist positive constants $c_3,c_4$
(depending only on~$c$, $\delta$, and the dimension) such that 
 for all~$u_{1,2}\in\Z^d$ with 
$(1+\delta)n< \|u_{1,2}\| < (c-\delta)n$ 
and $\|u_1-u_2\|\geq \delta n$
it holds that 
$c_3n^{-(d-2)} \leq G_{A_n}(u_1,u_2) \leq c_4 n^{-(d-2)}$.
\end{lem}
\begin{proof}
Indeed, we first notice that Proposition~4.6.2 of~\cite{LL10}
(together with the estimates on the Green's function
and the potential kernel, Theorems~4.3.1 of~\cite{LL10}
and~\eqref{formula_for_a})
 imply that 
$G_{A_n}(v,u_2) \asymp n^{-(d-2)}$ for all $d\geq 2$, 
where $\delta' n -1<\|v-u_2\|\leq \delta' n$,
 and~$\delta'\leq \delta$ is a small enough constant.
Then, use the fact that from any~$u_1$ as above,
the simple random walk comes from~$u_1$ to $\B(u_2,\delta' n)$
without touching~$\partial A_n$
with uniformly positive probability.
\end{proof}

\begin{lem}
\label{l_G_entrance}
Let~$d\geq 2$.
Let~$c,\delta,A_n$ be as in Lemma~\ref{l_G_annulus},
and assume that $(1+\delta)n\leq \|x\|\leq (c-\delta)n$,
$u\in\partial\B(n)$. Then, for some positive constants $c_5,c_6$
(depending only on~$c$, $\delta$, and the dimension) we have
\begin{equation}
\label{c/n_Poisson_general}
\frac{c_5}{n^{d-1}} \leq  H_{A_n}(x,u) \leq \frac{c_6}{n^{d-1}}.
\end{equation}
\end{lem}
Observe that, since $\IP_x[\tau_1(\B(n))<\tau_1(\partial\B(cn))]$
is bounded away from~$0$ and~$1$, the above result also
holds for the harmonic measure~$\hm_{\B(n)}(\cdot)$
(notice that the harmonic measure is a linear combination
of conditional entrance measures).
\begin{proof}
This can be proved essentially in the same way as in
 Lemma~6.3.7 of~\cite{LL10}. 
Namely, denote $B=\partial\B(n)\cup \partial\B((1+\delta)n)$ 
and use Lemma~6.3.6 of~\cite{LL10}
together with Lemmas~\ref{l_annulus_escape}
and~\ref{l_G_annulus} to write
(with $c_2=c_2(\delta)$, as in Lemma~\ref{l_annulus_escape})
\begin{align*}
 H_{A_n}(x,u) &= \sum_{z\in \partial\B((1+\delta)n)} G_{A_n}(z,x) 
  \IP_u\big[S_{\tau_1(B)}=z\big]\\
 & \leq c_4 n^{-(d-2)}\sum_{z\in \partial\B((1+\delta)n)}
\IP_u\big[S_{\tau_1(B)}=z\big]\\
 & \leq c_4 n^{-(d-2)} \times \frac{c_2}{n},
\end{align*}
obtaining the upper bound in~\eqref{c/n_Poisson_general}.
The lower bound is obtained in the same way
(using the lower bound on~$G_{A_n}$ from Lemma~\ref{l_G_annulus}).
\end{proof}

\begin{lem}
\label{l_escape_to}
Let $k>1$ and $x\in\partial \B(n)$.
Then, as $n\to\infty$ 
%(and uniformly in~$k$)
\begin{equation}
\label{escape_to_k}
 \IP_x\big[\tau_1(\partial\B(k+n))< \tau_1(\B(n))\big] 
= 
  \begin{cases}
    \displaystyle\frac{\hm_{\B(n)}(x)}
       {\frac{2}{\pi}\ln\big(1+\frac{k}{n}\big)+O(n^{-1})},
       & \text{for }d=2,\\
    \displaystyle\frac{\capa(\B(n))\hm_{\B(n)}(x)}
    {1-\big(1+\frac{k}{n}\big)^{-(d-2)}+O(n^{-1})}
     \vphantom{\int\limits^{A^B}}, & \text{for }d\geq 3.
  \end{cases}
\end{equation}
\end{lem}
We stress that the $O$'s in the above expressions depend only on~$n$,
not on~$k$.
\begin{proof}
% Let us first assume that $k\leq n$; afterwards we will
% explain how to generalize it to arbitrary~$k$.
Consider first the case $d\geq 3$. 
It is enough to prove it for the case $k\leq n^2/2$,
since for $k>n^2/2$ the second term in the denominator
is already~$O(n^{-1})$.
Now, Proposition~6.4.2 of~\cite{LL10} implies that,
for any $x\in \partial\B(n)$ and $m>n$
\[
 \Es_{\B(n)}(x) = \capa(\B(n))\hm_{\B(n)}(x)
 = \IP_x\big[\tau_1(\partial\B(m))<\tau_1(\B(n))\big]
\Big(1-O\Big(\frac{n^{d-2}}{m^{d-2}}\Big)\Big),
\]
% Let us first assume that $k\leq n$; afterwards we will
% explain how to generalize it to arbitrary~$k$.
so
\begin{equation}
\label{escape_to_n2}
 \IP_x\big[\tau_1(\partial\B(n^2))<\tau_1(\B(n))\big]
   = \capa(\B(n))\hm_{\B(n)}(x)\big(1+O(n^{-(d-2)})\big).
\end{equation}
On the other hand, with~$\nu$ being the entrance 
measure to~$\partial\B(n+k)$ starting
 from~$x$ and conditioned on the 
event $\big\{\tau_1(\partial\B(n+k))<\tau_1(\B(n))\big\}$,
we write using Lemma~\ref{l_exit_balls}~(ii)
\begin{align*}
 \lefteqn{\IP_x\big[\tau_1(\partial\B(n^2))<\tau_1(\B(n))\big]}\\
 &=
   \IP_x\big[\tau_1(\partial\B(n+k))<\tau_1(\B(n))\big]
   \IP_\nu\big[\tau_1(\partial\B(n^2))<\tau_1(\B(n))\big]\\
  &= \IP_x\big[\tau_1(\partial\B(n+k))<\tau_1(\B(n))\big]
   \Big(1-\Big(1+\frac{k}{n}\Big)^{-(d-2)}+O(n^{-1})\Big),
\end{align*}
and this, together with~\eqref{escape_to_n2},
implies~\eqref{escape_to_k} in higher dimensions.

Now, we deal with the case $d=2$. 
Assume first that $k\leq n^2/2$. Let~$y$
be such that $n^3<\|y\|\leq n^3+1$; also,
denote $A'=(\B(n^5)\setminus\B(n))\cup \partial\B(n)$. 
For any $z\in\partial\B(n^2)$ we can write
using Proposition~4.6.2~(b) together Lemma~\ref{l_exit_balls}~(i)
(starting from~$z$,
the walk reaches~$\B(n)$ before~$\B(n^5)$ with probability
$\frac{3}{4}(1+O(n^{-1}))$)
\begin{align}
G_{A'}(z,y) &=  \big(1+O(n^{-1})\big)
\Big(\frac{3}{4}\times \frac{2}{\pi} \ln n^3
 + \frac{1}{4}\times \frac{2}{\pi}\ln n^5
 - \frac{2}{\pi}\ln n^3\Big)\nonumber\\
 &= \frac{1}{\pi}\big(1+O(n^{-1})\big)\ln n.
\label{GA'}
\end{align}
% Abbreviate $\mu(z)=H_{\B(n^3)}(0,z)$ for $z\in\partial\B(n^3)$.
Next, Lemma~6.3.6 of~\cite{LL10}
% together with Proposition~6.4.5 of~\cite{LL10}
 implies that 
\begin{align}
 H_{A'}(y,x) &= \sum_{z\in\partial\B(n^2)}
  G_{A'}(z,y) \IP_x\big[S_{\tau_1(\partial\B(n^2))}=z,
   \tau_1(\partial\B(n^2))<\tau_1(\B(n))\big]\nonumber\\
 &= \IP_x\big[\tau_1(\partial\B(n^2))<\tau_1(\B(n))\big]
\sum_{z\in\partial\B(n^2)}
  G_{A'}(z,y) \mu(z) ,
% \Big(1+O\Big(\frac{\ln n}{n^2}\Big)\Big).
\label{eqL636}
\end{align}
where $\mu$ is the entrance measure to~$\partial\B(n^2)$ starting
 from~$x$, conditioned on the 
event $\big\{\tau_1(\partial\B(n^2))<\tau_1(\B(n))\big\}$.
% To deal with the summation in the right-hand side of~\eqref{eqL636},
% fix some~$v$ such that $n^2<\|v\|\leq n^2+1$, and observe
% that, again by Proposition~6.4.5 of~\cite{LL10},
% for any $z\in\partial\B(n^3)$
% \[
% G 
% \]

Then, by~(31) of~\cite{CPV15} (observe that 
Lemma~\ref{l_exit_balls}~(i) implies that, starting from~$y$,
the walk reaches~$\B(n)$ before~$\B(n^5)$ with probability
$\frac{1}{2}(1+O(n^{-1}))$)
we have
\begin{equation}
\label{Phm}
 H_{A'}(y,x) = \frac{1}{2}\hm_{\B(n)}\big(1+O(n^{-1})\big).
\end{equation}
So, from~\eqref{GA'}, \eqref{eqL636}, and~\eqref{Phm}
we obtain that
\begin{equation}
\label{goton2}
\IP_x\big[\tau_1(\partial\B(n^2))<\tau_1(\B(n))\big]
   = \frac{\hm_{\B(n)}(x)}{\frac{2}{\pi}\ln n}
 \big(1+O(n^{-1})\big).
\end{equation}
% \begin{equation}
% \label{form1_LL}
%  \IP_x[\tau_1(\partial\B(m))<\tau_1(\B(n))]
%    = \frac{\hm_{\B(n)}(x)}{\frac{2}{\pi}\ln m}
%  \Big(1+O\Big(\frac{n\ln m}{m}\Big)\Big),
% \end{equation}
% for all $m\geq 4n$. 
Let~$\nu$ be the entrance measure to~$\partial\B(n+k)$ starting
 from~$x$, conditioned on the 
event $\big\{\tau_1(\partial\B(n+k))<\tau_1(\B(n))\big\}$.
 Using~\eqref{goton2}, we write
\begin{align*}
 \lefteqn{\IP_x\big[\tau_1(\partial\B(n+k))<\tau_1(\B(n))\big]
  \IP_\nu\big[\tau_1(\partial\B(n^2))<\tau_1(\B(n))\big]}\\
 &= \IP_x\big[\tau_1(\partial\B(n^2))<\tau_1(\B(n))\big]
\phantom{*****************}\\
&= \frac{\hm_{\B(n)}(x)}{\frac{2}{\pi}\ln n}
 \big(1+O(n^{-1})\big).
\end{align*}
Since, by Lemma~\ref{l_exit_balls}~(i) we have
% \begin{equation}
% \label{n+k->n2}
\[
 \IP_\nu\big[\tau_1(\partial\B(n^2))<\tau_1(\B(n))\big]
   = \frac{\ln\big(1+\frac{k}{n}\big) + O(n^{-1})}{\ln n},
\]
% \end{equation}
this proves~\eqref{escape_to_k}
in the case $d=2$ and $k\leq n^2/2$.

% Now, let us explain the case of arbitrary $k\geq n$.
% (but that would lead to more complicated notations
% and the proofs will become less intuitive......)
% \textbf{(leave it like this, or write it for all $k$?)}
The case $k > n^2/2$ is easier:
just repeat \eqref{GA'}--\eqref{goton2} with~$k$
on the place of~$n^2$ (so that $n^3$ becomes~$k^{3/2}$
and $n^5$ becomes~$k^{5/2}$).
This concludes the proof of Lemma~\ref{l_escape_to}.
\end{proof}

Let us now come back to the specific case of $d=2$.
We need some facts regarding the conditional walk~$\s$.

\begin{lem}
\label{l_escape_from_ball}
Let $d=2$ and assume that $x\notin \B(y,r)$ and $\|y\|>2r\geq 1$. 
We have 
\begin{equation}
\label{eq_escape_from_ball}
 \IP_x\big[\htau_1(\B(y,r))<\infty\big]
= \frac{\big(a(y)+O(\|y\|^{-1}r)\big)\big(a(y)+a(x)-a(x-y)
+O(r^{-1})\big)}{a(x)\big(2a(y)-a(r)+O(r^{-1}+\|y\|^{-1}r)\big)}.
\end{equation}
\end{lem}
\begin{proof}
 This is Lemma~3.7 (i) of~\cite{CPV15}.
\end{proof}

\begin{lem}
\label{l_cap_distantball}
 Let $d=2$ and assume that $\|y\|>2r\geq 1$. 
We have
\begin{equation}
\label{eq_cap_distball}
 \capa\big(\{0\}\cup \B(y,r)\big)
  = \frac{\big(a(y)+O(\|y\|^{-1}r)\big)\big(a(y)+O(r^{-1})\big)}
{2a(y)-a(r)+O(r^{-1}+\|y\|^{-1}r)}.
\end{equation}
\end{lem}
\begin{proof}
 This is Lemma~3.9 (i) of~\cite{CPV15}.
\end{proof}

Then, we show that the walks~$S$ and~$\s$ are
 almost indistinguishable on a ``distant'' (from the origin) set.
For $A \subset \Z^2$, let~$\Gamma^{(x)}_A$ be the set of all 
 finite nearest-neighbour
trajectories that start at~$x\in A\setminus\{0\}$ 
and end when entering~$\partial A$ for the first time.
 For~$V\subset \Gamma^{(x)}_A$ write
 $S\in V$ if there exists~$k$ such that 
$(S_0,\ldots,S_k)\in V$ (and the same for
the conditional walk~$\s$). 

\begin{lem}
\label{l_relation_S_hatS}
 Assume that $V\subset \Gamma^{(x)}_A$
and suppose that $0\notin A$, and denote $s=\dist(0,A)$,
$r=\diam(A)$. Then, for $x \in A$,
\begin{equation}
\label{eq_relation_S_hatS2}
\IP_x[S\in V]
 =\IP_x\big[\s \in V\big]\Big(1+O\Big(\frac{r}{s\ln s}\Big)\Big).
\end{equation}
\end{lem}

\begin{proof}
 This is Lemma~3.3 (ii) of~\cite{CPV15}.
\end{proof}

\subsection{Excursions and soft local times}
\label{s_SLT}
Let~$X=(X_k, k\geq 0)$ be a simple random walk on the 
two-dimensional torus $\Z^2_n=\Z^2/n\Z^2$.
In this section we will develop some tools for
dealing with excursions of two-dimensional random 
interlacements and random walks on tori;
in particular, one of our goals 
is to construct a coupling between the set of
RI's excursions and the set of excursions of the 
% simple random 
walk~$X$ on the torus.

First, if $A\subset A'$ are (finite) subsets
of~$\Z^2$ or~$\Z^2_n$, then the excursions
between~$\partial A$ and~$\partial A'$ are
pieces of nearest-neighbour trajectories
that begin on~$\partial A$ and end on~$\partial A'$,
see Figure~\ref{f_excurs_both}, 
which is, hopefully, self-explanatory.  We refer  
to Section~3.4 of~\cite{CPV15} for formal definitions.
Here and in the sequel we denote by $(Z^{(i)}, i\geq 1)$
the (complete) excursions of the walk~$X$ 
between~$\partial A$ and~$\partial A'$, 
and by $(\hZ^{(i)}, i\geq 1)$ the $\RI$'s excursions
between~$\partial A$ and~$\partial A'$
(dependence on $n,A,A'$ is not indicated in these
notations when there is no risk of confusion).
\begin{figure}
\begin{center}
\includegraphics{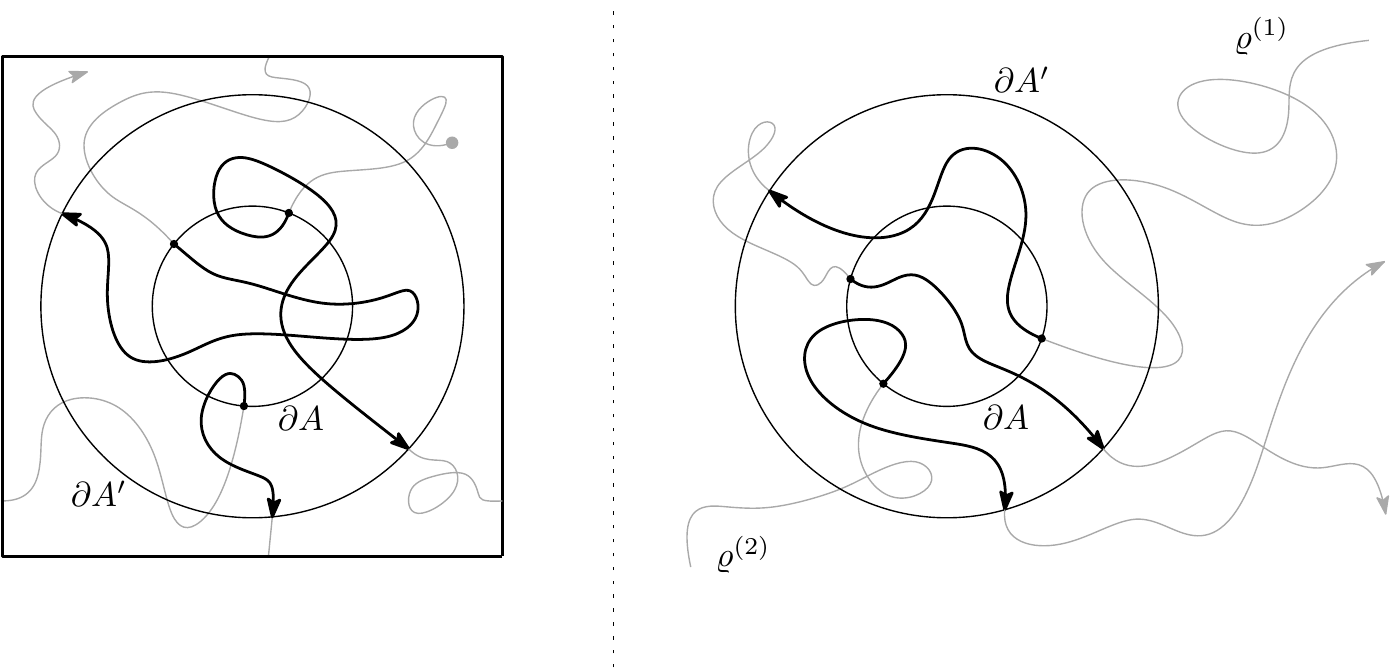}
\caption{Excursions (pictured as bold pieces of trajectories)
for simple random walk on the torus (on the left), and
random interlacements (on the right). Note the walk
``jumping'' from right side of the square to the left one,
and from the bottom one to the top one (the torus is
pictured as a square).
For random interlacements,
two trajectories, $\vr^{1,2}$, intersect the set~$A$; the first
trajectory produces two excursions, and the second only one.}
\label{f_excurs_both}
\end{center}
\end{figure}

Now, assume that we want to construct the excursions
of~$\RI(\alpha)$, say,
between~$\partial\B(y_0,n)$ and~$\partial\B(y_0,cn)$ for some $c>1$
and~$y_0\in\Z^2$. Also, (let us identify the torus $\Z^2_{n_1}$
with the square of size~$n_1$ centered in the origin of~$\Z^2$)
we want to construct the excursions
of the simple random walk on the torus~$\Z^2_{n_1}$ between~$\partial\B(y_0,n)$ and~$\partial\B(y_0,cn)$,
where~$n_1>n+1$.
It turns out that one may build both sets 
of excursions simultaneously on the same probability space, 
in such a way that, typically, most of the excursions 
are present in both sets (obviously, after a
translation by~$y_0$). 
This is done using the \emph{soft local times} method;
 we refer to Section~4 of~\cite{SLT}
for the general theory (see also Figure~1 of~\cite{SLT}
which gives some quick insight on what is going on),
and also to Section~2 of~\cite{CGPV13}. Here,
we describe the soft local times approach in a less formal way.
Assume, for definiteness, that 
 we want to construct the simple random walk's excursions
on~$\Z^2_{n_1}$, between~$\partial A$ and~$\partial A'$, and
suppose that the starting point~$x_0$ of the walk~$X$
does not belong to~$A$.  

We first describe our approach for the case of the torus.
For $x\notin A$ and~$y\in\partial A$ let us 
%denote~$\phi(x,y)=\IP_x[X_{\tau_1(A)}=y]$. 
denote
$$\phi(x,y)=\IP_x[X_{\tau_1(A)}=y].$$ 
For an excursion~$Z$ let~$\iota(Z)$ be the first
point of this excursion, and
$\ell(Z)$ be the last 
one; by definition, $\iota(Z)\in\partial A$
 and $\ell(Z)\in\partial A'$.
Clearly, for the random walk on the torus,
the sequence $\big((\iota(Z^{(j)}), \ell(Z^{(j)})), 
j\geq 1 \big)$
is a Markov chain with transition probabilities
\[
 P_{(y,z),(y',z')} = \phi(z,y')
   \IP_{y'}[X_{\tau_1(\partial A')}=z'].
\]

Now, consider a \emph{marked}
 Poisson point process on~$\partial A\times \R_+$
with rate~$1$. The (independent) marks are the simple random walk
trajectories started from the first coordinate of the Poisson
points (i.e., started at the corresponding site of~$\partial A$)
and run until hitting~$\partial A'$.
Then (see Figure~\ref{f_SLT_picture}; observe that~$A$ and~$A'$
need not be necessarily connected, as shown on the picture)
\begin{itemize}
 \item let~$\xi_1$ be the a.s.\ unique positive number such that
there is only one point of the Poisson process on the 
graph of $\xi_1\phi(x_0,\cdot)$ and nothing below;
 \item the mark of the chosen point is the first excursion 
(call it~$Z^{(1)}$) that we obtain;
 \item then, let~$\xi_2$ be the a.s.\ 
unique positive number such that
the graph of $\xi_1\phi(x_0,\cdot)+\xi_2\phi(\ell(Z^{(1)}),\cdot)$
contains only one point of the Poisson process, and there 
is nothing between this graph and the previous one;
 \item  the mark $Z^{(2)}$ of this point is our second excursion;
 \item and so on. 
\end{itemize}
It is possible to show 
%(see the above references) 
that the sequence of excursions obtained in this way  
 indeed has the same law as the simple random walk's excursions
 (in particular, conditional on~$\ell(Z^{(k-1)})$,
 the starting point of $k$th excursion is indeed
 distributed according to $\phi(\ell(Z^{(k-1)}),\cdot)$);
moreover, the $\xi$'s are i.i.d.\ random variables 
with Exponential($1$) distribution.
\begin{figure}
\begin{center}
\includegraphics[width=0.959\textwidth]{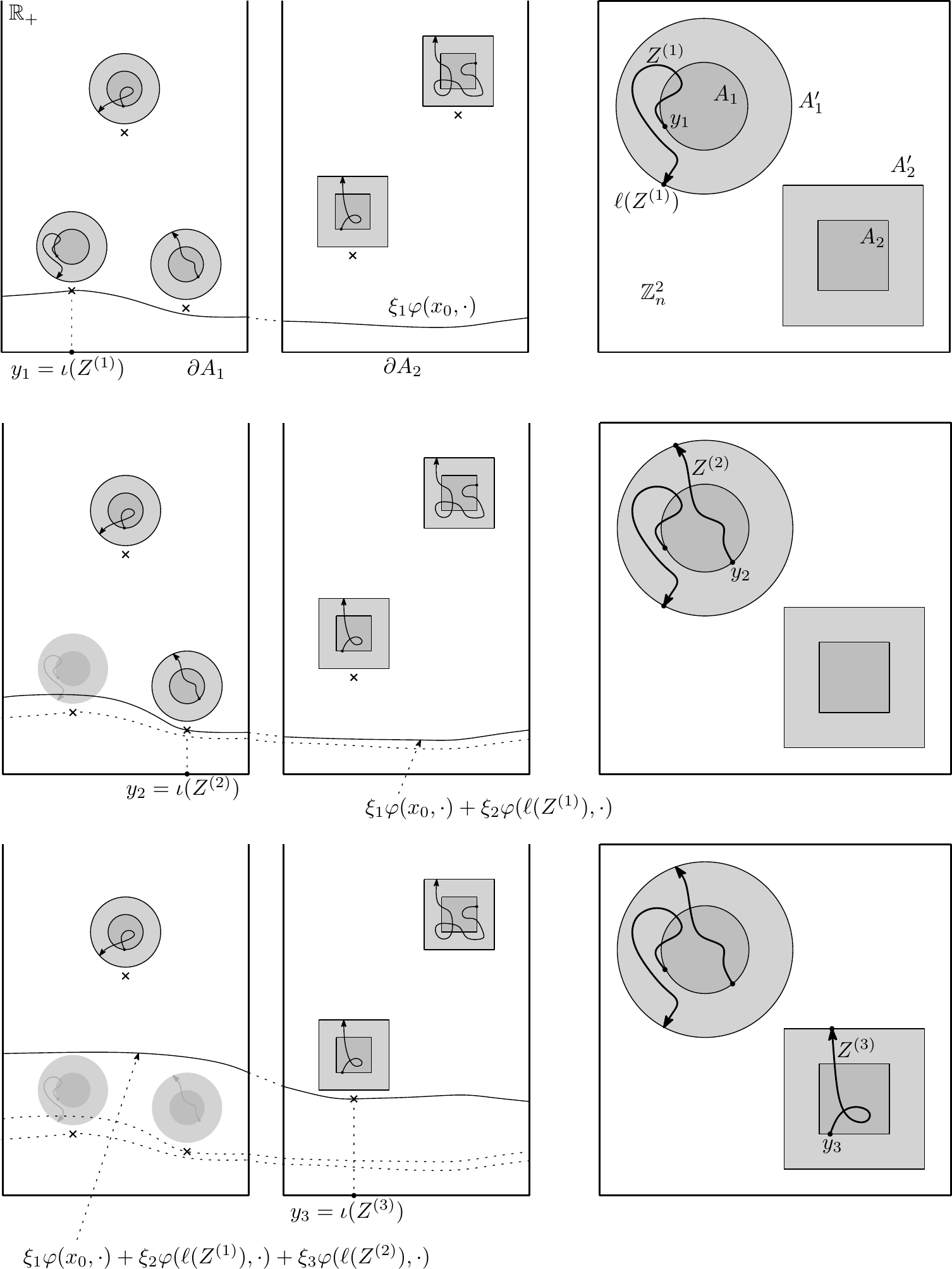}
\caption{Construction of the first three excursions
between~$\partial A$ and~$\partial A'$ on the torus~$\Z^2_n$
using the soft local times 
(here, $A=A_1\cup A_2$ and $A'=A'_1\cup A'_2$)}
\label{f_SLT_picture}
\end{center}
\end{figure}

So, let us denote by $\xi_1, \xi_2,\xi_3, \ldots$  a sequence
of i.i.d.\ random variables with Exponential distribution
with parameter~$1$. 
According to the above informal description,
the soft local time of $k$th excursion
is a random vector indexed by $y\in\partial A$, 
defined as follows:
\begin{equation}
\label{df_SLT_torus}
 L_k(y) = \xi_1\phi(x_0,y) 
   + \sum_{j=2}^k \xi_j \phi(\ell(Z^{(j-1)}),y).
\end{equation}

For the random interlacements, the soft local times are
defined analogously. Recall that $\hhm_A$ defines
the (normalized) harmonic measure on~$A$ with respect to 
the $\s$-walk.
For $x\notin A$ and~$y\in\partial A$ let 
\begin{equation}
\label{df_hat_phi}
 \hphi(x,y) = \IP_x\big[\s_{\htau_1(A)}=y,\htau_1(A)<\infty\big]
   + \IP_x[\htau_1(A)=\infty] \hhm_A(y).
\end{equation}
% (that is, $\hphi(x,y)$ is the probability that the $\s$-walk
% starting from~$x$
% enters~$A$ through~$y$, ).
Analogously, for the random interlacements,
the sequence $\big((\iota(\hZ^{(j)}), \ell(\hZ^{(j)})), j\geq 1 \big)$
is also a Markov chain, with transition probabilities
\[
 {\widehat P}_{(y,z),(y',z')} = \hphi(z,y')
   \IP_{y'}\big[\s_{\htau_1(\partial A')}=z'\big].
\]
The process of picking the excursions for the random 
interlacements is quite analogous: if the last 
excursion was~$\hZ$, we use the 
probability distribution $\hphi(\ell(\hZ),\cdot)$
to choose the starting point of
the next excursion.   Clearly, the last term
in~\eqref{df_hat_phi} is needed for~$\hphi$ to have total mass~$1$;
informally, if the $\s$-walk from~$x$ does not ever hit~$A$,
we just take the ``next'' trajectory of the random
interlacements that does hit~$A$, and extract the
excursion from it (see also~(4.10) of~\cite{CT14}). 
Again, let $\hat\xi_1, \hat\xi_2,\hat\xi_3, \ldots$ be a sequence
of i.i.d.\ random variables with Exponential distribution
with parameter~$1$. 
Then, define the soft local time of random interlacement
 of $k$th excursion as
\begin{equation}
\label{df_SLT_RI}
 \hL_k(y) = \hat\xi_1\hphi(x_0,y) 
   + \sum_{j=2}^k \hat\xi_j \hphi(\ell(\hZ^{(j-1)}),y).
\end{equation}

Define the following two measures on~$\partial A$,
one for the random walk on the torus, and the other for random 
interlacements:
\begin{align}
\hm_A^{A'}(y) &= \IP_y[\tau_1(\partial A')<\tau_1(A)]
\Big(\sum_{z\in\partial A}\IP_z[\tau_1(\partial A')
<\tau_1(A)]\Big)^{-1},
\label{df_hm_set_RW}\\
\hhm_A^{A'}(y) &= \IP_y[\htau_1(\partial A')<\htau_1(A)]
\Big(\sum_{z\in\partial A}
\IP_z[\htau_1(\partial A')<\htau_1(A)]\Big)^{-1}.
\label{df_hm_set_RI}
\end{align}
Informally, these are ``harmonic measures with respect to~$A'$'';
the ``real'' harmonic measures would be recovered as~$A'$
expands towards the whole space~$\Z^2$.
Similarly to Lemma~6.1 of~\cite{CT14} 
%\textbf{(cite also that proposition from~\cite{Szn12}?)},
 one can obtain the following important facts:
the measure 
\begin{equation}
\label{df_psi}
\psi(y,z) = \hm_A^{A'}(y)\IP_{y}[X_{\tau_1(\partial A')}=z]
\end{equation}
is invariant for the Markov chain $(\iota(Z^{(j)}), \ell(Z^{(j)}))$,
and the measure
\begin{equation}
\label{df_hpsi}
 \hpsi(y,z) = \hhm_A^{A'}(y)
           \IP_{y}\big[\s_{\htau_1(\partial A')}=z\big]
\end{equation}
is invariant for the Markov chain 
$(\iota(\hZ^{(j)}), \ell(\hZ^{(j)}))$.
%\textbf{(check this for RI's!)}.
Notice also that $\hm_A^{A'}$ and~$\hhm_A^{A'}$ are the 
marginals of the stationary measures for the entrance
points (i.e., the first coordinate of the Markov chains).
In particular, this implies that, almost surely, 
\[
 \lim_{k\to\infty} \frac{L_k(y)}{k} = \hm_A^{A'}(y)
\quad \text{and} \quad
 \lim_{k\to\infty} \frac{\hL_k(y)}{k} = \hhm_A^{A'}(y),
\]
for any $y\in\partial A$.

The next result is needed to have a control on 
the large and moderate deviation 
probabilities for soft local times.
\begin{lem}
\label{l_LD_SLT}
Let $\gamma_2>\gamma_1>1$ be some fixed constants,
assume that $n>3(\gamma_1-1)^{-1}$, and
abbreviate $n_1=\gamma_2 n$. For the random walk on the 
torus~$\Z^2_{n_1}$, abbreviate $A=\B(n)$ and $A'=\B(\gamma_1 n)$.
For the random interlacements,  
abbreviate $B=\B(y_0, n)$ and $B'=\B(y_0, \gamma_1 n)$,
where~$y_0\in\Z^2$ is such that $\|y_0\|\geq 2\gamma_1 n$.
Then there exist positive constants $c,c_1,c_2$ such that
for all~$k\geq 2$ and all $\theta\in(0,(\ln k)^{-1})$ we have
\begin{align}
\IP\Big[
% \frac{1}{\sqrt{k}}
\sup_{y\in\partial A}
\big|L_k(y)-k\hm_{A}^{A'}(y)\big| 
\geq \frac{c\sqrt{k}+\theta k}{n_1}\Big] 
\leq c_1 e^{-c_2 \theta^2 k}, \label{eq_LD_SLT_rw}\\
\IP\Big[
% \frac{1}{\sqrt{k}}
\sup_{y\in\partial B}
\big|\hL_k(y)-k\hhm_B^{B'}(y)\big| 
\geq \frac{c\sqrt{k}+\theta k}{n}\Big] 
\leq c_1 e^{-c_2 \theta^2 k}.
\label{eq_LD_SLT_RI}
\end{align}
% for all large enough~$n$.
\end{lem}

\begin{proof}
We prove only~\eqref{eq_LD_SLT_rw}, the proof
of~\eqref{eq_LD_SLT_RI} is completely analogous.
 Due to Lemma~\ref{l_G_entrance}, 
it is enough to show that for some (sufficiently large) $c',c'_1$
and (sufficiently small) $c'_2>0$ 
\begin{equation}
\label{LD_modif_SLT}
 \IP\Big[
% \frac{1}{\sqrt{k}}
\sup_{y\in\partial A}
\Big|\frac{L_k(y)}{\hm_{A}(y)}
  -k\frac{\hm_{A}^{A'}(y)}{\hm_{A}(y)}\Big| 
\geq c'\sqrt{k}+\theta k\Big] 
\leq c'_1 e^{-c'_2 \theta^2 k}.
\end{equation}

We prove the above inequality in the following way:
\begin{enumerate}
 \item Using renewals, we split the sequence of excursions 
into independent blocks.
 \item We show that controlling that sequence of i.i.d.\ 
blocks of excursions is enough to be able to control
the original sequence.
 \item Then, we obtain an upper bound on the expectation
of the sum of independent blocks. For this, we estimate the 
bracketing entropy integral and then use an inequality due to
Pollard.
 \item Finally, we control the deviation (from the expectation) 
probabilities, using a suitable concentration inequality.
\end{enumerate}

\medskip 
\noindent
\textbf{Step 1.}
Again, Lemma~\ref{l_G_entrance} implies that 
there exists~$\lambda>0$ such that for all 
$x\in \partial A'$ and~$y\in \partial A$
we have $\phi(x,y)\geq 2\lambda \hm_{A_1}^{A'_1}(y)$.
Consider a sequence of random variables
 $\eta_1, \eta_2, \eta_3, \ldots$, independent of everything,
and such that $\IP[\eta_j=1]=1-\IP[\eta_j=0]=\lambda$ for all~$j$.
For $j\geq 1$ define $\rho_j=m$ iff $\eta_1+\cdots+\eta_m=j$
and~$\eta_m=1$ (that is, $\rho_j$ is the position of $j$th ``$1$''
in the $\eta$-sequence).
The idea is that we force the Markov chain to
have renewals at times when $\eta_\cdot = 1$,
and then try to approximate the soft local time
by a sum of independent random variables. More
precisely, assume that $\ell(Z^{(j-1)})=x$. Then, 
we choose the starting point~$\iota(Z^{(j)})$ of the $j$th excursion
in the following way
\[
 \iota(Z^{(j)}) \sim 
   \begin{cases}
    \displaystyle\frac{1}{1-\lambda}
      \big(\phi(x,\cdot)-\lambda \hm_{A_1}^{A'_1}\big),
          & \text{if }\eta_j=0,\\
     \hm_{A_1}^{A'_1},\phantom{\int\limits^{A^B}}
      & \text{if }\eta_j=1.
   \end{cases}
\]
Denote $W_0(\cdot) = L_{\rho_1-1}(\cdot)$, and
\[
 W_j(\cdot) = L_{\rho_{j+1}-1}(\cdot) - L_{\rho_j-1}(\cdot)
\]
for $j\geq 1$. By construction, it holds that 
$(W_j,j\geq 1)$ is a sequence of i.i.d.\ random vectors.
Also, it is straightforward to obtain that
$\IE W_j(y) = \lambda^{-1}\hm_{A_1}^{A'_1}(y)$
for all $y\in\partial A$ and all $j\geq 1$.

% To avoid cumbersome notations, in the calculations
% below we pretend that~$\lambda k$ is integer
% (one can even make that rigorous by choosing a rational~$\lambda$
% and then going along a subsequence of $k$'s).

\medskip 
\noindent
\textbf{Step 2.}
Now, we are going to show that, to prove~\eqref{LD_modif_SLT},
it is enough to prove that, for some positive constants 
$c'',c''_1, c''_2$
\begin{equation}
\label{LD_iid_SLT}
  \IP\Big[
% \frac{1}{\sqrt{m}}
\sup_{y\in\partial A}
\Big|\frac{\sum_{j=1}^m W_j(y)}{\hm_{A}(y)}
  -\lambda^{-1}m\frac{\hm_{A}^{A'}(y)}
{\hm_{A}(y)}\Big| 
\geq c''\sqrt{m}+\theta m\Big] 
\leq c''_1 e^{-c''_2 \theta^2 m}.
\end{equation}
for all~$m\geq 2$ and all $\theta\in(0,(\ln m)^{-1})$.
Observe that we can assume without loss of generality
that $c''_1\leq \frac{1}{2}$; indeed, if the above holds
with \emph{some} $c''_1>0$, then, by increasing~$c''$
(and, possibly, decreasing~$c''_2$;
 note that 
$(c''+h)\sqrt{m}+\theta m = c''\sqrt{m}+(\theta +\frac{h}{\sqrt{m}})m$)
we can put an arbitrarily small constant before the exponent 
in the right-hand side.

Abbreviate 
\begin{align*}
 R_k &= \sup_{y\in\partial A}
\Big|\frac{L_k(y)}{\hm_{A}(y)}-
k\frac{\hm_{A}^{A'}(y)}{\hm_{A}(y)}\Big|,\\
\intertext{and} 
{\widetilde R}_k &= \sup_{y\in\partial A}
\Big|
\frac{\sum_{j=1}^k W_j(y)}{\hm_{A}(y)}
  -\lambda^{-1} k\frac{\hm_{A}^{A'}(y)}{\hm_{A}(y)}\Big|.
\end{align*}
Let us first show that~\eqref{LD_iid_SLT} implies
\begin{equation}
\label{LD_max}
 \IP\Big[\max_{i\in [m,2m]} {\widetilde R}_i 
\geq 4c''\sqrt{m}+5\theta m\Big]
  \leq 2c''_1 e^{-3c''_2 \theta^2 m}.
\end{equation}
For this, define the random variable
\[
 N = \min\big\{i\in [m,2m]: {\widetilde R}_i 
\geq 4c''\sqrt{m}+5\theta m\big\}
\]
(by definition, $\min \emptyset := +\infty$),
so the left-hand side of~\eqref{LD_max}
is equal to $\IP\big[N\in [m,2m]\big]$.
Note also that the right-hand side of~\eqref{LD_iid_SLT}
does not exceed~$\frac{1}{2}$ (recall 
that we assumed that $c''_1\leq \frac{1}{2}$).
Now, \eqref{LD_iid_SLT} implies (note that $\sqrt{3}<4-\sqrt{2}$)
\begin{align*}
 c''_1 e^{-3 c''_2 \theta^2 m} &\geq 
 \IP\big[{\widetilde R}_{3m} \geq c''\sqrt{3m}+3\theta m\big]
\\
 &\geq \sum_{j=m}^{2m} \IP[N=j]
  \IP\big[{\widetilde R}_{3m-j} < c''\sqrt{2m}+2\theta m\big]
\\
 &\geq \IP\big[N\in [m,2m]\big] \times
\min_{j\in [m,2m]}   
 \IP\big[{\widetilde R}_{3m-j} < c''\sqrt{3m-j}+\theta (3m-j)\big]
\\
 &\geq \frac{1}{2} \IP\big[N\in [m,2m]\big]
\end{align*}
for all large enough~$m$. 

Next, let us denote $\sigma_k = \min\{j\geq 1: \rho_j > k\}$.
By~\eqref{df_hm_set_RW} and Lemma~\ref{l_escape_to}
we may assume that $\frac{1}{2}\leq\frac{\hm_{A}^{A'}(y)}
{\hm_{A}(y)}\leq 2$, and, due to Lemma~\ref{l_G_entrance}, 
$\frac{L_k(y)}{\hm_{A}(y)}\leq \tilde{c}$ for
some $\tilde{c}>0$.
% 
% the centered soft local time is a sum 
% of terms of the form $a_j\xi_j-b_j$, where $a_j,b_j$
% are uniformly bounded)
%  it holds that
% , uniformly in~$y$
% \begin{equation}
% \label{RtildeR}
% - c  \sum_{j=1}^{\rho_{1}-1} \xi_j
% - c  \sum_{j=\rho_{\lambda k}}^{k+1} \xi_j
%  \leq {\widetilde R}_y-R_y
%  \leq c \sum_{j=k+1}^{\rho_{\lambda k}} \xi_j
% \end{equation}
% (we assume that the sum equals~$0$ in case the upper limit
% is strictly less than the lower limit in the summation).
% So, we obtain that 
% \begin{equation}
% \label{dominate_Xi}
%  \sup_{y\in\partial A}
% \Big|\frac{L_k(y)}{\hm_{A}(y)}
%   -k\frac{\hm_{A}^{A'}(y)}{\hm_{A}(y)}\Big| 
%   \leq \sup_{y\in\partial A}
% \Big|\frac{\sum_{j=1}^{\sigma_k}W_j(y)}{\hm_{A}(y)}
%   -\lambda^{-1}\sigma_k\frac{\hm_{A}^{A'}(y)}
% {\hm_{A}(y)}\Big| + \Xi,
% \end{equation}
% where the random variable~$\Xi$ is stochastically
% bounded above by $\sum_{j=1}^{\rho_1-1}\xi_j$.
% \[
%  c \sum_{j=1}^{|k-\rho_{\lambda k}|+\rho_1} \xi'_k,
% \]
% where $\xi'_1,\xi'_2,\xi'_3,\ldots$ are i.i.d.\
% Exponential($1$) random variables.
% Using this, it is elementary (\textbf{really?}) to prove that
% \begin{equation}
% \label{LD_Xi}
% \IP[\Xi \geq c\sqrt{k}+\theta k] 
% \leq c''_1 e^{-c''_2 \theta^2 k}.
% \end{equation}
So, we can write
\begin{equation}
\label{dominate_Xi}
 R_k \leq {\widetilde R}_{\sigma_k} + 2|\lambda^{-1}\sigma_k-k|
   +\tilde{c}\sum_{i=k+1}^{\rho_{\sigma_k}}\xi_i.
\end{equation}
Now, observe that $\sigma_k-1$ is a Binomial$(k,\lambda)$
random variable, and $\rho_{\sigma_k} - k$ is 
Geometric$(\lambda)$. Therefore,
the last two terms in the right-hand side of~\eqref{dominate_Xi}
are easily dealt with; that is, we may write
for large enough~$\hat{c}>0$
\begin{align}
\IP\big[2|\lambda^{-1}\sigma_k-k| \geq \hat{c}\sqrt{k}+\theta k\big]
 &\leq c_4  e^{-c'_4 \theta^2 k},
\label{LD_simple}\\
\IP\Big[\tilde{c}\sum_{i=k+1}^{\rho_{\sigma_k}}\xi_i\geq \theta k\Big] 
&\leq e^{-c_5 \lambda \theta k}.
\label{LD_very_simple}
\end{align}
Then, using~\eqref{LD_max} together
with \eqref{LD_simple}--\eqref{LD_very_simple}, we obtain
(recall~\eqref{LD_modif_SLT})
\begin{align*}
 \lefteqn{\IP\Big[R_k\geq \big(4c''(\textstyle\frac{2\lambda}{3})^{1/2}
  + \hat{c}\big)\sqrt{k} + \big(\textstyle\frac{10}{3}\lambda+2\big)
\theta k\Big]}\\
  &\leq \IP\Big[\max_{i\in[\frac{2}{3}\lambda k,\frac{4}{3}\lambda k]}
{\widetilde R}_i \geq 4c''(\textstyle\frac{2\lambda}{3})^{1/2}
 \sqrt{k} + 5\theta\cdot \textstyle\frac{2}{3}\lambda k\Big]
%\\
% & \quad 
+ \IP\big[\sigma_k\notin [\frac{2}{3}\lambda k,
\frac{4}{3}\lambda k]\big] \\
& \quad 
+\IP\big[2|\lambda^{-1}\sigma_k-k| \geq \hat{c}\sqrt{k}+\theta k\big]
+\IP\Big[\tilde{c}\sum_{i=k+1}^{\rho_{\sigma_k}}\xi_i\geq \theta k\Big] 
\\
&\leq 2c''_1 e^{-2\lambda c''_2 \theta^2 k}
 + e^{-c_6\lambda k} + c_4  e^{-c'_4 \theta^2 k}
 + e^{-c_5 \lambda \theta k},
\end{align*}
and this shows that
 it is indeed enough for us to prove~\eqref{LD_iid_SLT}.

% The last term in the right-hand side of~\eqref{dominate_Xi}
% is easily dealt with, and, since we clearly have that
% $\IP[\lambda k/2 \leq \sigma_k\leq k]\geq 1-e^{-c_3 k}$,
%  we just use the union bound in~\eqref{dominate_Xi} 
% to obtain that it is indeed enough for us to prove~\eqref{LD_iid_SLT}.
% \textbf{(can we get rid of the union bound? it's ugly\dots)}

\medskip 
\noindent
\textbf{Step 3.}
Now, the advantage of~\eqref{LD_iid_SLT} is that 
we are dealing with i.i.d.\ random vectors there,
so it is convenient to use some machinery from the 
theory of empirical processes.    
First, the idea is to use the Pollard
inequality (cf.~(1.2) of~\cite{VW11}) to prove that
\begin{equation}
\label{empirical_expect}
 \IE {\widetilde R}_k \leq c_7\sqrt{k}
\end{equation}
for some $c_7>0$ (note that the above estimate is uniform 
with respect to the size of $\partial A$). 
To use the language of empirical processes, we are dealing 
here with random elements of the form $\tW_j=\frac{W_j}{\hm_A}$ which are 
positive vectors indexed by sites of~$\partial A$.
Let also~$Y$ be a generic positive vector indexed by sites of~$\partial A$.
For $y\in\partial A$ 
let~$\EE_y$ be the evaluation functional at~$y$: $\EE_y(Y):=Y(y)$.
Denote by~$\F=\{\EE_y, y\in\partial A\}$ the class of functions we
are interested in; then, we need to find an upper
bound on the expectation of
$\sup_{f\in\F} |\sum_{j=1}^k f(\tW_j) - \IE f(\tW_j)|$.
Using the terminology of~\cite{VW11}, let $\|f\|_2:=\sqrt{\IE f^2(\tW)}$,
where~$\tW$ has the same law as the $\tW_j$'s above.  
Consider the \emph{envelope function}~$F$ defined by
\[
 F(Y) = \sup_{y\in \partial A} \EE_y(Y) = \sup_{y\in \partial A}Y(y).
\]
% ($Y$ is a generic positive vector indexed by sites of~$\partial A$).
Due to Lemma~\ref{l_G_entrance}, we have
\begin{equation}
\label{|F|_2}
 \|F\|_2 \leq c_8.
\end{equation}

%To be able to apply~(1.2) of~\cite{VW11}, one has to estimate
Let us define 
the \emph{bracketing entropy integral}
\begin{equation}
\label{entropy_integral}
 J_{[\,]}\big(1,\F,\|\cdot\|_2\big) 
    = \int_0^1 \sqrt{1+\ln N_{[\,]}(s\|F\|_2, \F,\|\cdot\|_2 )}ds.
\end{equation}
In the above expression,
$N_{[\,]}(\delta, \F,\|\cdot\|_2 )$ is the so-called
\emph{bracketing number}: the minimal number of
brackets $[f,g]=\{h: f\leq h\leq g\}$ needed to cover~$\F$
 of size~$\|g-f\|_2$ smaller than~$\delta$.
We now recall~(1.2) of~\cite{VW11}:
\begin{equation}
\label{Pollard_ineq}
\IE {\widetilde R}_k \leq c J_{[\,]}(1,\F,\|\cdot\|_2) 
\|F\|_2 \sqrt{k};
\end{equation}
so, to prove~\eqref{empirical_expect}, 
we need to obtain an upper bound on the bracketing entropy integral.

Let us define ``arc intervals'' on~$\partial A$
by $I(y,r)=\{z\in \partial A: \|y-z\|\leq r\}$, where $y\in\partial A, r>0$.
Observe that $I(y,r)=\{y\}$ in case $r<1$.
Define
\[
 f^{y,r}(Y) = \inf_{z\in I(y,r)}Y(z), \quad 
    g^{y,r}(Y) = \sup_{z\in I(y,r)}Y(z);
\]
in order to cover~$\F$,
we are going to use brackets of the form $[f^{y,r},g^{y,r}]$.
Notice that if $z\in I(y,r)$ then $\EE_z\in [f^{y,r},g^{y,r}]$,
so a covering of~$\F$ by the above brackets corresponds
to a covering of~$\partial A$ by ``intervals'' $I(\cdot, \cdot)$.
Let us estimate the size of the bracket $[f^{y,r},g^{y,r}]$;
it is here that Proposition~\ref{p_SRW_Hoelder} comes into play.
We have
\begin{align}
 \|g^{y,r}-f^{y,r}\|_2 & 
 = \textstyle\sqrt{\IE\big|\!\sup_{z\in I(y,r)}\tW(z)
   - \inf_{z\in I(y,r)}\tW(z)\big|}
 \nonumber\\
   &\leq c_9 r^\beta \|\xi_1+\cdots+\xi_\rho\|_2
   \nonumber\\
    &\leq 2c_9\lambda^{-1} r^\beta;
\label{ExpGeomExp}
\end{align}
in the above calculation, $\rho$ is a Geometric random variable
with success probability~$\lambda$, $\xi$'s are i.i.d.\ Exponential(1)
random variables also independent of~$\rho$, 
and we use an elementary fact that $\xi_1+\cdots+\xi_\rho$
is then also Exponential with mean $\lambda^{-1}$.

Then, recall~\eqref{|F|_2}, and observe that, for any~$\delta>0$
it is possible to cover~$\F$ with $|\partial A|=O(n_1)$
brackets of size smaller than~$\delta$ 
(just cover each site separately with brackets 
$[f^{\cdot,1/2},g^{\cdot,1/2}]$ of zero size).
That is, for any $s>0$ it holds that 
\begin{equation}
\label{est_brackets_arb}
 N_{[\,]}\big(s\|F\|_2, \F,\|\cdot\|_2 \big) \leq c_{10} n_1.
\end{equation}
Next, if $s\geq c_{11} n_1^{-\beta}$, then we are able to 
use intervals of size $r=O(n_1 s^{1/\beta})$ to cover~$\partial A$,
so we have
\begin{equation}
\label{est_brackets_s>const}
 N_{[\,]}\big(s\|F\|_2, \F,\|\cdot\|_2 \big) = O(n_1/r) 
\leq c_{12} s^{-1/\beta}.
\end{equation}
So (recall~\eqref{entropy_integral}) the bracketing entropy integral
can be bounded above by
\[
 c_{11} n_1^{-\beta}\sqrt{1+\ln (c_{10}n_1)} 
  + \int_{c_{11} n_1^{-\beta}}^1\sqrt{1+\ln (c_{12}s^{-1/\beta})}\, ds
    \leq c_{13}.
\]
Then, using~\eqref{Pollard_ineq}, we obtain~\eqref{empirical_expect}.

\medskip 
\noindent
\textbf{Step 4.}
Here, let us use Theorem~4 of~\cite{A08} to prove that
(with $t=\theta k$)
\begin{equation}
\label{empirical_deviation}
 \IP\big[{\widetilde R}_k \geq 2\IE {\widetilde R}_k + t\big]
  \leq c_{14} e^{-c_{15}t^2/k} + c_{16} e^{-c_{17}t/\ln k};
\end{equation}
this is enough for us since, 
due to the assumption $\theta < (\ln k)^{-1}$
it holds that
\[
 c_{14} e^{-c_{15}t^2/k} + c_{16} e^{-c_{17}t/\ln k}
   \leq c_{18}e^{-c_{19}\theta^2 k}.
\]
To apply that theorem, we only need 
to estimate the $\psi_1$-Orlicz norm
of $\EE_y(\tW)$, see Definition~1 of~\cite{A08}. 
But (recall the notations just below~\eqref{ExpGeomExp})
it holds that $\EE_y(\tW)$ is stochastically bounded above
by $const\times$Exponential($\lambda$) random variable,
so the $\psi_1$-Orlicz norm is uniformly bounded above\footnote{a
straightforward calculation shows that the $\psi_1$-Orlicz norm
of an Exponential random variable equals its expectation}.
% Also, observe that, in this case, the second term in the
% right-hand side of the third display of Theorem~4 of~\cite{A08}
% is of smaller order than the first term.
The factor $\ln k$ in the last term in the right-hand
side of~\eqref{empirical_deviation} comes from the Pisier's
inequality, cf.~(13) of~\cite{A08}.

Finally, combining~\eqref{empirical_expect} and~\eqref{empirical_deviation},
we obtain~\eqref{LD_iid_SLT}, and, as observed before,
this is enough to conclude the proof 
of Lemma~\ref{l_LD_SLT}.
\end{proof}

Next, we need  a fact that one may call the \emph{consistency}
of soft local times. Assume that we need to construct
excursions of some process (i.e., random walk, random interlacements,
or just independent excursions) between~$\partial A$
and~$\partial A'$; let~$(\hL_k(y), y\in\partial A)$
be the soft local time of $k$th excursion (of random interlacements,
for definiteness).
On the other hand, we may be interested in
simultaneously constructing 
the excursions also between~$\partial A_1$
and~$\partial A'_1$, where $A'\cap A'_1=\emptyset$
and~$A_1\subset A'_1$.  Let
$(\hL^*_k(y), y\in\partial A)$ be the soft local
time at the moment when $k$th excursion between~$\partial A$
and~$\partial A'$ was chosen in this latter
construction. We need the following
simple fact:
\begin{lem}
\label{l_consistency_SLT}
It holds that 
\[
 \big(\hL_k(y), y\in\partial A\big) 
         \eqlaw \big(\hL^*_k(y), y\in\partial A\big)
\]
for all~$k\geq 1$.
\end{lem}
\begin{proof}
First, due to the memoryless property of the 
Poisson process,  it is clearly enough to prove that
$\hL_1\eqlaw\hL^*_1$. This, by its turn, can be easily 
obtained from the fact that $\hZ^{(1)} \eqlaw \hZ^{(1),*}$,
where $\hZ^{(1)}$ and $\hZ^{(1),*}$ are the first excursions
between~$\partial A$ and~$\partial A'$
chosen in both constructions.
\end{proof}

Also, we need to be able to control the number of excursions~$N^*_t$
up to time~$t$ on the torus~$\Z^2_n$ between~$\partial\B(\gamma_1n)$
and~$\partial\B(\gamma_2n)$, $\gamma_1<\gamma_2<1/2$. 
%Define $N^*_t$ to be    (depends on~$n_1$)
\begin{lem}
\label{l_number_exc_torus}
 For all large enough $n$, all $t\geq n^2$
 and all $\delta >0$ we have
\begin{equation}
\label{eq_number_exc_torus} 
 \IP\Big[(1-\delta)\frac{\pi t}{2n^2 \ln(\gamma_2/\gamma_1)}
\leq N^*_t 
\leq (1+\delta)\frac{\pi t}{2n^2 \ln(\gamma_2/\gamma_1)}\Big] 
\geq 1- c_1 \exp\Big(- \frac{c_2\delta^2 t}{n^2}\Big),
\end{equation}
where $c_{1,2}$ are positive constants depending on $\gamma_{1,2}$.
\end{lem}
\begin{proof}
Note that there is a much more general result on the large
deviations of the excursion counts for the Brownian motion
(the radii of the concentric disks need not be of order~$n$),
see Proposition~8.10 of~\cite{BK14}. So, we give the proof 
of Lemma~\ref{l_number_exc_torus} in a rather sketchy way.
First, let us rather work with the \emph{two-sided} stationary version
of the walk $X=(X_j, j\in\Z)$ (so that~$X_j$ is uniformly distributed
on~$\Z^2_n$ for any~$j\in\Z$). For $x\in\partial\B(\gamma_1 n)$ 
define the set 
\begin{align*}
 J_x &= \big\{k\in \Z : X_k=x, \text{ there exists }i<k 
\text{ such that }
   X_i\in \partial\B(\gamma_2 n) \\
   & \qquad\qquad\qquad \qquad\qquad\text{ and }
    X_m\in \B(\gamma_2 n)\setminus 
      (\B(\gamma_1 n)\cup\partial\B(\gamma_2 n)) 
       \text{ for } i<m<k\big\},
\end{align*}
and let $J=\bigcup_{x\in \partial\B(\gamma_1 n)} J_x$.
Now, Lemma~\ref{l_escape_to} together with
the reversibility argument used in Lemma~6.1 of~\cite{CT14}
imply that 
\[
 \IP[0\in J_x] = \IP[X_0=x]
 \IP_x\big[\tau_1(\partial\B(\gamma_2 n))<\tau_1(\B(\gamma_1 n))\big]
      =  n^{-2}\frac{\hm_{\B(\gamma_1 n)}(x)}{\frac{2}{\pi}
        \ln \frac{\gamma_2}{\gamma_1} + O(n^{-1})},
\]
so (since $\hm_{\B(\gamma_1 n)}(\cdot)$ sums to~$1$)
\begin{equation}
\label{0_in_J}
 \IP[0\in J] = \Big(\frac{2n^2}{\pi}
        \ln \frac{\gamma_2}{\gamma_1} + O(n)\Big)^{-1}.
\end{equation}
Let us write $J=\{\sigma_j, j\in\Z\}$, 
where $\sigma_{-1}<0$, $\sigma_0\geq 0$,
and $\sigma_j<\sigma_{j+1}$ for all $j\in\Z$.
As noted just after \eqref{df_hm_set_RW}--\eqref{df_hm_set_RI},
the invariant entrance measure to~$\B(\gamma_1 n)$ for excursions
is $\nu =\hm_{\B(\gamma_1 n)}^{\B(\gamma_2 n)}$.
Let $\IE_\nu^*$ be the expectation for the walk
with~$X_0\sim\nu$ and conditioned on~$0\in J$ (that
is, for the \emph{cycle-stationary} version of the walk).
 Then, in a standard way one obtains from~\eqref{0_in_J} that
\begin{equation}
\label{expect_sig-sig}
 \IE_\nu^* \sigma_1
  = \IE_\nu^* (\sigma_1-\sigma_0)
    = \frac{2n^2}{\pi}
        \ln \frac{\gamma_2}{\gamma_1} + O(n).
\end{equation}
% where $ \hm_{\B(\gamma_1 n)}^{\B(\gamma_2 n)}$.
Note also that in this setup
(radii of disks of order~$n$) it is easy to control
the tails of $\sigma_1-\sigma_0$ since in each interval 
of length~$O(n^2)$ there is at least one complete excursion
with uniformly positive probability (so
there is no need to apply the Khasminskii's lemma\footnote{see
e.g.\ the argument between~(8.9) and~(8.10) of~\cite{BK14}}, 
as one usually does for proving results on large
deviations of excursion counts).
To conclude the proof of Lemma~\ref{l_number_exc_torus},
it is enough to apply a renewal argument similar to the one
used in the proof of Lemma~\ref{l_LD_SLT}
(and in Section~8 of~\cite{BK14}).
\end{proof}

\section{Proofs of the main results}
\label{s_proofs}

\begin{proof}[\textbf{Proof of Proposition~\ref{p_SRW_Hoelder}}]
Fix some $x,y,z$ as in the statement of the proposition. 
We need the following fact:
\begin{lem}
\label{l_GH_revers}
We have
\begin{equation}
\label{revers_yz}
 H_{A_n}(x,u) = \IE_u \sum_{j=1}^{\tau_1(\partial A_n)}
    \1{S_j=x} 
= \frac{1}{2d}\sum_{\substack{v\sim u:\\v\in A_n\setminus \partial A_n}}
  G_{A_n}(v,x)
\end{equation}
(that is, $H_{A_n}(x,u)$ equals
the mean number of visits to~$x$ before hitting~$\partial A_n$,
starting from~$u$)
for all $u\in \partial\B(n)$.
\end{lem}

\begin{proof}
This follows from a standard reversibility argument.
Indeed, write (the sums below are over all nearest-neighbor 
trajectories beginning in~$x$ and ending in~$u$ that do not 
touch~$\partial A_n$ before entering~$u$; $\vr^*$ stands for
$\vr$ reversed, $|\vr|$ is the number of edges in~$\vr$,
 and $k(\vr)$ is the number of times~$\vr$
was in~$x$)
\begin{align*}
  H_{A_n}(x,u) &= \sum_{\vr}  (2d)^{-|\vr|}\\
 &= \sum_{\vr}  (2d)^{-|\vr^*|}\\
 & = \sum_{j=1}^\infty \sum_{\vr: k(\vr)=j}(2d)^{-|\vr^*|},
\end{align*}
and observe that the $j$th term in the last line 
is equal to the probability  that~$x$ is visited
at least~$j$ times (starting from~$u$) before hitting~$\partial A_n$.
This implies~\eqref{revers_yz}.
\end{proof}

Note that, by Lemma~\ref{l_G_entrance} we have also 
\begin{align}
\frac{c_1}{n^{d-1}} \leq  H_{A_n}(x,u) \leq \frac{c_2}{n^{d-1}},
\label{c/n_Poisson}\\
\intertext{and, as a consequence (since $\hm_{\B(n)}(u)$
is a convex combination in~$x'\in \partial\B((1+2\eps)n)$ of $H_{A_n}(x',u)$)}
\frac{c_1}{n^{d-1}} \leq  \hm_{\B(n)}(u) \leq \frac{c_2}{n^{d-1}}
\label{c/n_harmonic}
\end{align}
for all $u\in \partial\B(n)$. Therefore, without
restricting generality we may assume that $\|y-z\|\leq (\eps/9)n$,
since if $\|y-z\|$ is of order~$n$, then~\eqref{eq_SRW_Hoelder}
holds for large enough~$C$.

So, using Lemma~\ref{l_GH_revers}, we can estimate the difference
between the mean numbers of visits to one fixed site
in the interior of the annulus starting from
two close sites at the boundary,
instead of dealing with hitting probabilities
of two close sites starting from that fixed site.

Then, to obtain~\eqref{eq_SRW_Hoelder}, we proceed in the following way. 
\begin{itemize}
 \item[(i)] Observe that, to go from a site~$u\in\partial\B(n)$ to~$x$, 
the particle needs to go first to~$\partial\B((1+\eps)n)$;
we then prove that the probability of that is ``almost'' proportional 
to~$\hm_{\B(n)}(u)$, see~\eqref{escape_to_1+eps}.
 \item[(ii)] In~\eqref{def_tilde_P} we introduce
% if $y,z\in\partial\B(n)$ are ``close'' to each other,
two walks conditioned on hitting~$\partial\B((1+\eps)n)$
before returning to~$\partial\B(n)$,
starting from $y,z\in\partial\B(n)$. The idea is that they
 will likely couple before reaching~$\partial\B((1+\eps)n)$.
 \item[(iii)] More precisely, we prove that each time the distance
between the original point on~$\partial\B(n)$ and
the current position of the (conditioned) walk is doubled,
there is a uniformly positive chance that the coupling
of the two walks succeeds (see the argument just after~\eqref{always_Harnack}).
% this gives rise to the term
% in the right-hand side of~\eqref{eq_SRW_Hoelder}.
 \item[(iv)] To prove the above claim, we define two sequences
 $(U_k)$ and $(V_k)$ of subsets of the annulus
 $\B((1+\eps)n)\setminus \B(n)$,
 as shown on Figure~\ref{f_cond_coupling}. Then, we prove that
 the positions of the two walks on first hitting of~$V_k$ can be 
 coupled with uniformly positive probability,
regardless of their positions on first hitting of~$V_{k-1}$.
For that, we need two technical steps:
 \begin{itemize}
 \item[(iv.a)] If one of the two conditioned walks hits~$V_{k-1}$
 at a site which is ``too close'' to $\partial\B(n)$ (look at the 
 point~$Z_{k-1}$ on Figure~\ref{f_cond_coupling}), we need to assure 
 that the walker can go ``well inside'' the set~$U_k$ with at least
constant probability (see~\eqref{escape_box_Psi}).
 \item[(iv.b)] If the (conditioned) walk is already ``well inside''
the set~$U_k$, then one can apply the Harnack inequality to prove 
that the exit probabilities are comparable 
in the sense of~\eqref{always_Harnack}.
 \end{itemize}
 \item[(v)] There are $O\big(\ln\frac{n}{\|y-z\|}\big)$ ``steps'' on the 
 way to~$\partial\B((1+\eps)n)$, and the coupling is successful
 on each step with uniformly positive probability. So, 
 in the end the coupling fails with probability polynomially
 small in $\frac{n}{\|y-z\|}$,
 cf.~\eqref{prob_Upsilon}.
 \item[(vi)] Then, it only remains to gather the pieces 
 together (the argument after~\eqref{exit_end}). The last technical
 issue is to show that, even if the coupling does not succeed,
 the difference of expected hit counts cannot be too large;
 this follows from~\eqref{capacity_ball_3}
and Lemma~\ref{l_G_annulus}.
\end{itemize}
%  From now on, we work in dimension $d=2$; in the end of the proof
% we discuss how to adapt it to $d\geq 3$.
We now pass to the detailed arguments.
By Lemma~\ref{l_escape_to} we have for any~$u\in\partial\B(n)$
\begin{equation}
\label{escape_to_1+eps}
 \IP_u\big[\tau_0(\partial\B((1+\eps)n))< \tau_1(\partial\B(n))\big] 
= 
  \begin{cases}
    \displaystyle\frac{\hm_{\B(n)}(u)}
       {\frac{2}{\pi}\ln(1+\eps)+O(n^{-1})},
       & \text{for }d=2,\\
    \displaystyle\frac{\capa(\B(n))\hm_{\B(n)}(u)}
    {1-(1+\eps)^{-(d-2)}+O(n^{-1})}
     \vphantom{\int\limits^{A^B}}, & \text{for }d\geq 3,
  \end{cases}
\end{equation}
so, one can already notice that the probabilities
to escape to~$\partial\B((1+\eps)n)$ normalized by the 
harmonic measures are roughly the same for all sites of~$\partial\B(n)$.
Define the events
\begin{equation}
\label{df_event_F}
F_j=\big\{\tau_0(\partial\B((1+\eps)n)) < \tau_j(\partial\B(n))\big\}
\end{equation}
for $j=0,1$.
For $v\in \B((1+\eps)n)\setminus \B(n)$ denote~$h(v) = \IP_v[F]$;
clearly, $h$ is a harmonic function inside 
the annulus $\B((1+\eps)n)\setminus \B(n)$, and the simple
random walk on the annulus conditioned on~$F_0$ is in fact
a Markov chain (that is, the Doob's $h$-transform of the simple
random walk) with transition probabilities
\begin{equation}
\label{def_tilde_P}
\tP_{v,w} = 
   \begin{cases}
    \displaystyle\frac{h(w)}{2d h(v)}, & v\in \B((1+\eps)n)\setminus 
         (\B(n)\cup \partial \B((1+\eps)n)), w\sim v,\\
       0, & \text{otherwise}.
   \end{cases}
\end{equation}
On the first step (starting at~$u\in\partial\B(n)$), the transition
probabilities of the conditioned walk
are described in the following way: the walk
goes to $v\notin\B(n)$ with probability
\[
 h(v)\Bigg(\sum_{\substack{v'\notin\B(n):\\ v'\sim u}}h(v')\Bigg)^{-1}.
\]

\begin{figure}
\begin{center}
\includegraphics{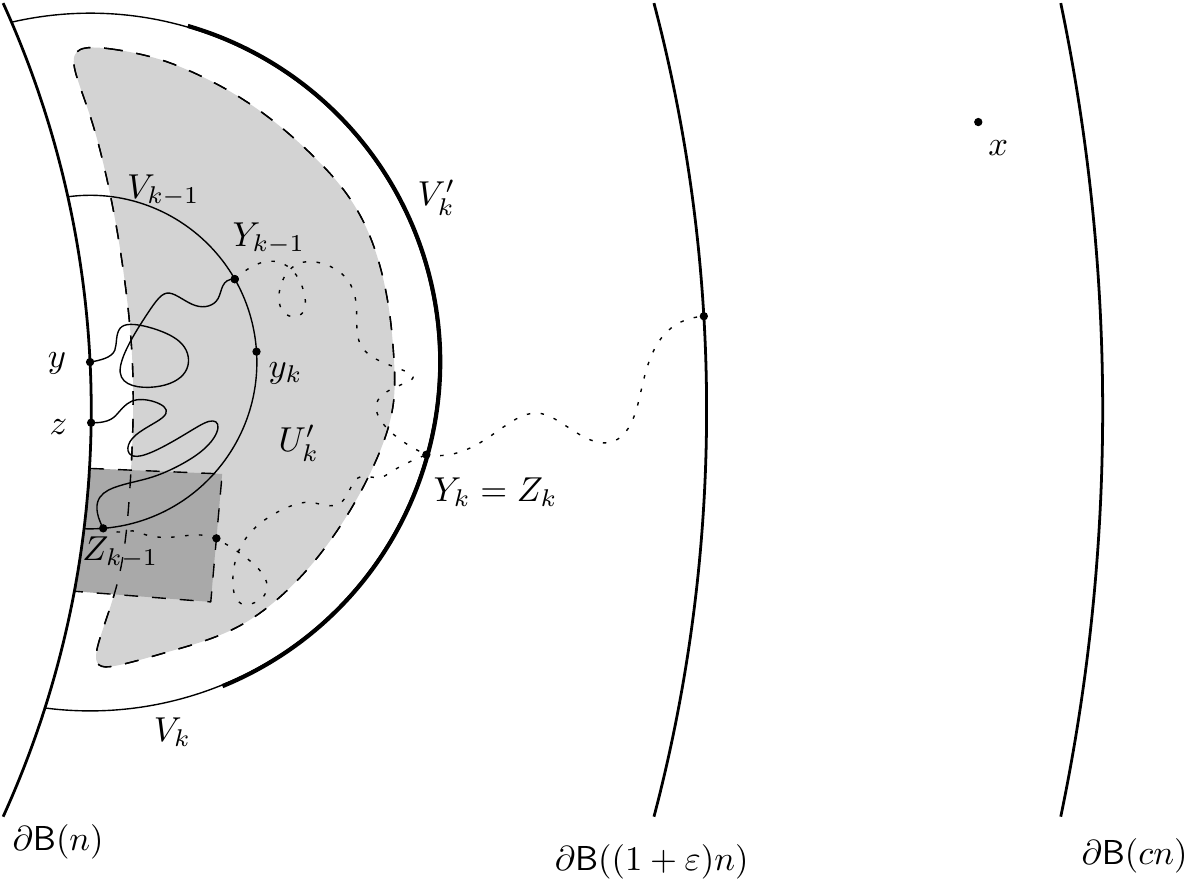}
\caption{On the coupling of conditioned walks
 in the proof of Proposition~\ref{p_SRW_Hoelder}. Here,
 $Y_{k-1}$ and $Z_{k-1}$ are positions of the walks started
 in~$y$ and~$z$, and we want to couple their exit points on~$V'_k$.
 The $y$-walk is already in~$U'_k$, but we need to force the $z$-walk
 to advance to~$U'_k$ in the set $\Psi(Z_{k-1},\|y-z\|2^{k-1})$
 (dark grey
 on the picture), so that the Harnack inequality would be applicable.}
\label{f_cond_coupling}
\end{center}
\end{figure}

Let $k_0=\max\{j: 3\|y-z\|2^j < \eps n\}$, and let us define
the sets 
\begin{align*}
 U_k &= \B(y, 3\|y-z\|2^k) \setminus 
 \big(\B(n)\setminus\partial\B(n)\big),\\
V_k &= \partial U_k \setminus \partial\B(n),
\end{align*}
for $k=1,\ldots, k_0$, see Figure~\ref{f_cond_coupling}.
Also, define~$y_k$ to be the closest integer point to 
$y+3\|y-z\|2^k\frac{y}{\|y\|}$.
Clearly, it holds that 
\begin{equation}
\label{k0>}
 k_0 \geq c_2\ln\frac{c_3\eps n}{\|y-z\|}.
\end{equation}
% Let $y',z'\in \B((1+\eps)n)\setminus \B(n)$ be neighbors
% of~$y$ and~$z$, correspondingly.
Denote by $\tS^{(1)}$ and $\tS^{(2)}$
the conditioned random walks started from~$y$ and~$z$.
For $k=1,\ldots,k_0$ denote $Y_k=\tS^{(1)}_{\ttau_1(V_k)}$,
$Z_k=\tS^{(2)}_{\ttau_1(V_k)}$,
where~$\ttau_1$ is the hitting time for the $\tS$-walks, 
defined as in~\eqref{hitting_t}.
The goal is to couple $(Y_1,\ldots,Y_{k_0})$ and 
$(Z_1,\ldots,Z_{k_0})$ in such a way that with high
probability there exists $k_1\leq k_0$ such that
$Y_j=Z_j$ for all $j=k_1,\ldots,k_0$;
we denote by~$\Upsilon$ the corresponding
coupling event. Clearly, this
generates a shift-coupling of~$\tS^{(1)}$ and $\tS^{(2)}$;
if we managed to shift-couple them before they 
reach~$\partial \B((1+\eps)n)$, then the number of visits
to~$x$ will be the same. 

For $v\in\R^d$ let~$\ell_v=\{r v, r>0\}$ be the ray
in $v$'s direction. Now, for any~$v$ with $n<\|v\|\leq (1+\eps)n$
and $s\in (0,\eps n)$ define the (discrete) set 
\[
 \Psi(v,s) = \big\{u\in \Z^d : n<\|u\|\leq n+s,
   \dist(u,\ell_v)\leq s/2\big\}. 
\]
Denote also by
\[
 \partial^+ \Psi(v,s) = \big\{u\in \partial\Psi(v,s) : 
      n+s-1<  \|u\|\leq n+s\big\}
\]
the ``external part'' of the boundary of $\Psi(v,s)$
(on Figure~\ref{f_cond_coupling}, it is the rightmost side 
of the dark-grey ``almost-square'').
Observe that, by Lemma~\ref{l_annulus_escape},
we have
\begin{equation}
\label{harmonic_bounds}
 c_5\frac{\|v\|-n+1}{n} \leq h(v)\leq c_6\frac{\|v\|-n+1}{n}.
\end{equation}
We need the following simple fact: if $\|v\|-n<2s$,
\begin{equation}
\label{escape_box_Psi}
 \IP_v\big[\tS_{{\tilde \tau}(\partial\Psi(v,s))}
   \in \partial^+ \Psi(v,s)\big] \geq c_7 
\end{equation}
for some positive constant~$c_7$.
To see that, it is enough 
to observe that the probability of the corresponding
event for the simple random walk~$S$ is $O(\frac{\|v\|-n+1}{s})$
(roughly speaking, the part transversal to~$\ell_v$ behaves as
a $(d-1)$-dimensional simple random walk, so it does not 
go too far from~$\ell_v$ with constant probability,
 and the probability that the projection on~$\ell_v$
``exits to the right'' is clearly $O(\frac{\|v\|-n+1}{s})$
by a gambler's ruin-type argument; 
or one can use prove
an analogous fact for the Brownian motion and then use the 
KMT-coupling). Now (recall~\eqref{def_tilde_P})
the weight of an $\tS$-walk trajectory
is its original weight divided by the value of~$h$
in its initial site, and multiplied by the value of~$h$
in its end. But (recall~\eqref{harmonic_bounds}), the value
of the former is $O(\frac{\|v\|-n+1}{n})$,
and the value of the latter is $O(\frac{s}{n})$.
Gathering the pieces, we obtain~\eqref{escape_box_Psi}.

Note also the following: let~$A$ be a subset
of $(\B((1+\eps)n)\setminus\B(n))\cup\partial\B(n)$, and
for $u\in A, v\in\partial A$
denote by $\tH_A(u,v) 
= \IP_u\big[\tS_{\ttau_1(\partial A)}=v\big]$
the Poisson kernel 
with respect to the conditioned walk~$\tS$.
Then, it is elementary to obtain that $\tH_A(u,v)$
is proportional to $h(v) H_A(u,v)$, that is
\begin{equation}
\label{tilde_H}
 \tH_A(u,v) = h(v) H_A(u,v) 
    \Big(\sum_{v'\in \partial A}h(v') H_A(u,v')\Big)^{-1}.
\end{equation}

Now, we are able to construct the coupling.
Denote by $V'_k= \{v\in V_k : \|v\|\geq n+3\|y-z\|2^{k-1}\}$
to be the ``outer'' part of~$V_k$ (depicted on 
Figure~\ref{f_cond_coupling}
as the arc with double thickness), and 
denote by $U'_k = \{u\in U_k : \dist(u,\partial U_k)\geq \|y-z\|2^{k-3}\}$
the ``inner region'' of~$U_k$. Using~\eqref{harmonic_bounds}
and~\eqref{tilde_H} together
with the Harnack inequality (see e.g.\ Theorem~6.3.9 of~\cite{LL10}), 
we obtain that, for some~$c_8>0$
\begin{equation}
\label{go_out_from_inside}
 \tH_{U_k}(u,v) \geq c_8 \tH_{U_k}(y_k,v)
\end{equation}
for all $u\in U'_k$ and $v\in V'_k$. The problem is that~$Z_{k-1}$
(or $Y_{k-1}$, or both)
may be ``too close'' to~$\partial\B(n)$, and so we need to ``force''
it into~$U'_k$ in order to be able to apply the Harnack inequality. 
First, from an elementary geometric argument
one obtains that, for any $v\in V_{k-1}\setminus U'_k$
\begin{equation}
\label{partial+Psi_inside}
\partial^+ \Psi(v,\|y-z\|2^{k-1})\subset U'_k.
\end{equation}
Then, \eqref{escape_box_Psi} and~\eqref{partial+Psi_inside}
together imply that indeed with uniformly positive 
probability an $\tS$-walk started from~$v$ enters~$U'_k$
before going out of~$U_k$. 
Using~\eqref{go_out_from_inside},
% Using~\eqref{tilde_H}
% and the Harnack inequality (e.g.\ Theorem~6.3.9 of~\cite{LL10}),
we then obtain that
\begin{equation}
\label{always_Harnack}
 \tH_{U_k}(u,v) \geq c_9 \tH_{U_k}(y_k,v)
\end{equation}
for all $u\in V_{k-1}$ and $v\in V'_k$. Also, it is clear 
 that $\sum_{v\in V'_k}\tH_{U_k}(y_k,v)$
is uniformly bounded below by a constant $c_{10}>0$,
 so on each step $(k-1)\to k$
the coupling works with probability at least $c_9c_{10}$.
Therefore, by~\eqref{k0>}, we can couple $(Y_1,\ldots,Y_{k_0})$ and 
$(Z_1,\ldots,Z_{k_0})$ in such a way that $Y_{k_0}=Z_{k_0}$
with probability at least 
$1-(1-c_9c_{10})^{k_0}=1-c_{11} \big(\frac{n}{\|y-z\|}\big)^{-\beta}$.

%  Lemma~\ref{l_G_annulus} implies that the mean number of visits
% to~$x$ is of the same order for any starting 
% point from~$\partial \B((1+\eps)n)$. 
% So, if the coupling fails, this entails at most a constant
% factor difference between the expected number of visits to~$x$
% for the two walks (that start from~$y$ and~$z$).

Now, we are able to finish the proof of Proposition~\ref{p_SRW_Hoelder}. 
Recall that we denoted by~$\Upsilon$ the coupling
event of the two walks (that start from~$y$ and~$z$);
as we just proved,
\begin{equation}
\label{prob_Upsilon}
  \IP[\Upsilon^\complement] 
       \leq c_{11}\Big(\frac{n}{\|y-z\|}\Big)^{-\beta}.
\end{equation}
Let $\nu_{1,2}$ be the exit measures of the two walks
on~$\partial \B((1+\eps)n)$.
 For $j=1,2$ we have for any $v\in\partial \B((1+\eps)n)$
\begin{equation}
\label{exit_end}
 \nu_j(v) = \IP[\Upsilon]\nu_*(v) 
     + \IP\big[\Upsilon^\complement\big]\nu'_j(v) ,
\end{equation}
where
\begin{align*}
\nu_*(v) &= \IP\big[\tS^{(j)}_{\ttau_1(\partial \B((1+\eps)n)}=v
    \mid \Upsilon\big] ,\\
\nu'_j(v) &= \IP\big[\tS^{(j)}_{\ttau_1(\partial \B((1+\eps)n)}=v
    \mid \Upsilon^\complement\big]
\end{align*}
(observe that if the two walks are coupled on hitting~$V_{k_0}$,
then they are coupled on hitting~$\partial \B((1+\eps)n)$,
so~$\nu_*$ is the same for the two walks).
For $u\in\partial \B(n)$
define the random variables (recall Lemma~\ref{l_GH_revers})
\[
 \G_u = \IE_u \sum_{j=1}^{\tau_1(\partial A_n)}
    \1{S_j=x},
\]
and recall the definition of the event~$F_1$ from~\eqref{df_event_F}.
% \[
%  F = \big\{\tau_1(\partial \B((1+\eps)n))<\tau_1(\partial\B(n))\big\}.
% \]
% for $j=1,2$.
We write, using~\eqref{escape_to_1+eps}
\begin{align*}
\lefteqn{ \frac{H_{A_n}(x,y)}{\hm_{\B(n)}(y)}
 - \frac{H_{A_n}(x,z)}{\hm_{\B(n)}(z)} }\\
&=
  \IE\Big(\frac{\G_y}{\hm_{\B(n)}(y)}
 - \frac{\G_z}{\hm_{\B(n)}(z)}\Big)\\
&= \Big(\frac{\IP_y[F_1]}{\hm_{\B(n)}(y)}
\big(\IP[\Upsilon]G(\nu_*,x) 
     + \IP\big[\Upsilon^\complement\big]G(\nu'_1,x)\big)\\
& \qquad\quad -  \frac{\IP_z[F_1]}{\hm_{\B(n)}(z)}
\big(\IP[\Upsilon]G(\nu_*,x) 
     + \IP\big[\Upsilon^\complement\big]G(\nu'_2,x)\big)\Big)\\
&\leq 
 \begin{cases}
  G(\nu_*,x)O(n^{-1}) + c_{12}\IP\big[\Upsilon^\complement\big]
%  \displaystyle\frac{\IP_y[F]}{\hm_{\B(n)}(y)}
G(\nu'_1,x),
   & \text{for }d=2,\\
  G(\nu_*,x)\capa(\B(n))O(n^{-1})
+ c_{13}\IP\big[\Upsilon^\complement\big]
       \capa(\B(n))G(\nu'_1,x),\vphantom{\int\limits^{A^B}}
   & \text{for }d\geq 3.\\
 \end{cases}
\end{align*}
Note that~\eqref{capacity_ball_3}
and Lemma~\ref{l_G_annulus} imply that, 
for \emph{any} probability measure~$\mu$
on~$\partial\B((1+\eps)n)$, it holds that
$G(\mu,x)$ (in two dimensions) and $\capa(\B(n))G(\mu,x)$
(in higher dimensions) are of constant order. 
Together with~\eqref{prob_Upsilon}, this implies that
\[
 \frac{H_{A_n}(x,y)}{\hm_{\B(n)}(y)}
 - \frac{H_{A_n}(x,z)}{\hm_{\B(n)}(z)}
 \leq c_{14}n^{-1} + c_{15} \Big(\frac{n}{\|y-z\|}\Big)^{-\beta}.
\]
Since~$y$ and~$z$ can be interchanged, 
this concludes the proof of Proposition~\ref{p_SRW_Hoelder}. 
\end{proof}

\smallskip

\begin{proof}[\textbf{Proof of Theorem~\ref{t_critical}}]
 Consider the sequence $b_k= \exp\big(\exp(3^k)\big)$, 
and let $v_k = b_ke_1 \in \R^2$.
Let us fix some $\gamma\in \big(1,\sqrt{\pi/2}\big)$.
Denote $B_k=\B(v_k,b_k^{1/2})$ and $B'_k=\B(v_k,\gamma b_k^{1/2})$.
Observe that Lemma~\ref{l_cap_distantball}
together with~\eqref{formula_for_a} imply
\begin{equation}
\label{cap_B_k}
 \capa\big(B_k\cup\{0\}\big) 
   = \frac{4}{3\pi}\big(1+O((\ln^{-1}b_k)\big)\ln b_k .
\end{equation}

 Let~$N_k$ be the number of excursions
between $\partial B_k$ and~$\partial B'_k$ in $\RI(1)$.
Lemma~\ref{l_escape_from_ball} implies that for 
any~$x\in\partial B'_k$ it holds that
\begin{equation}
\label{hit_B_k}
 \IP_x[\htau(B_k)<\infty] = 1 - \frac{2\ln\gamma}{3\ln b_k}
\big(1+O((\ln^{-1}b_k)\big),
\end{equation}
so the number of excursions of one particle
has ``approximately Geometric'' distribution
with parameter $\frac{2\ln\gamma}{3\ln b_k}(1+O((\ln^{-1}b_k))$.
Observe that if~$X$ is a Geometric$(p)$ random variable
and~$Y$ is Exponential$(\ln (1-p)^{-1})$ random variable, then 
$Y\preceq X \preceq Y+1$, where``$\preceq$'' means
stochastic domination.  So, the number of excursions of one particle
dominates an Exponential$\big(\frac{2\ln\gamma}{3\ln b_k}
(1+O(\ln^{-1}b_k))\big)$ and is dominated by 
Exponential$\big(\frac{2\ln\gamma}{3\ln b_k}
(1+O(\ln^{-1}b_k))\big)$ plus~$1$.
% Now, we need to show that for any $C>0$ there exists $h_C>0$
% suth that
% \begin{equation}
% \label{step_sqrt_below}
%  \IP\Big[N_k\leq \frac{2}{\ln\gamma}\ln^2 b_k 
%    - C\ln^{3/2} b_k\Big]\geq h_C
% \end{equation}
% for all~$k$. Indeed, clearly, $N_k$ is stochastically
% upper bounded by a compound Poisson random variable
% \begin{equation}
% \label{compound_Poisson}
%  \sum_{k=1}^\zeta (1+\eta_k),
% \end{equation}
% where~$\zeta$ is Poisson with parameter
% $\frac{4}{3\pi}(1+O(b_k^{-1/2}))\ln b_k$, and
% $\eta$'s are i.i.d.\ Exponential random variables with rate
% $r_k\frac{2\ln\gamma}{3\ln b_k}(1+O(b_k^{-1/2}))$.
% Clearly, for any given~$C'>0$ with uniformly positive
% probability it holds that $\zeta\leq
% \frac{4}{3\pi}\ln b_k - C'\ln^{1/2}b_k$.
% Observe that,
% by Theorem~1 of~\cite{CR86}, the median of a Gamma-distributed
% random variable is bounded above by its mean.
% Since, given~$\zeta=n$, the random variable 
% in~\eqref{compound_Poisson}
% becomes $n+\text{Gamma}(n,r_k)$, 
% the claim~\eqref{step_sqrt_below} follows.

Now, let us argue that 
\begin{equation}
\label{Nk_Normal}
\frac{\ln\gamma}{\sqrt{6}\ln^{3/2}b_k}
\Big(N_k - \frac{2}{\ln\gamma}\ln^2 b_k\Big)
 \convlaw \text{standard Normal.}
\end{equation}
Indeed, for the (approximately) compound Poisson
random variable~$N_k$ the previous discussion yields
\begin{equation}
\label{compound_Poisson}
 \sum_{k=1}^\zeta \eta_k \preceq 
\frac{N_k}{\ln b_k} \preceq \sum_{k=1}^{\zeta'} (\eta'_k+\ln^{-1}b_k),
%\qquad \text{a.s.,}
\end{equation}
where~$\zeta$ and $\zeta'$ are both Poisson with parameter
$\frac{4}{3}(1+O((\ln^{-1}b_k))\ln b_k$ (the difference is 
in the $O(\cdot)$), and
$\eta$'s are i.i.d.\ Exponential random variables with rate
$\frac{2\ln\gamma}{3}(1+O(\ln^{-1}b_k))$.
Since the Central Limit Theorem is clearly valid
for $\sum_{k=1}^\zeta \eta_k$ (the expected number of 
terms in the sum goes to infinity, while the number of summands
remain the same), one obtains~\eqref{Nk_Normal}
after some easy calculations\footnote{Indeed, 
if~$Y_\lambda=\sum_{j=1}^{Q_\lambda} Z_j$ is a compound Poisson
random variable, where $Q_\lambda$ is Poisson with mean~$\lambda$
and~$Z$'s are i.i.d.\ Exponentials with parameter~$1$, then
a straightforward computation shows that the moment
generating function of $(2\lambda)^{-1/2}(Y_\lambda-\lambda)$
is equal to $\exp(\frac{t^2}{2(1-(t/2\lambda))})$,
which converges to $e^{t^2/2}$ as $\lambda\to \infty$.}.

Next, observe that $\frac{\pi}{4\gamma^2}>\frac{1}{2}$ by our 
choice of $\gamma$. Choose some~$\beta\in (0,\frac{1}{2})$
in such a way that $\beta + \frac{\pi}{4\gamma^2}> 1$, 
and define~$q_\beta>0$ to be such that 
\[
 \int_{-\infty}^{-q_\beta}\frac{1}{\sqrt{2\pi}}e^{-x^2/2}\, dx = \beta. 
\]
Define also the sequence of events
\begin{equation}
\label{df_Phi_k}
 \Phi_k = \Big\{N_k\leq \frac{2}{\ln\gamma}\ln^2 b_k - 
    q_\beta\frac{\sqrt{6}\ln^{3/2}b_k}{\ln\gamma}\Big\}.
\end{equation}
Now, the goal is to prove that
\begin{equation}
\label{mnogo_Phi_k}
 \liminf_{n\to\infty} \frac{1}{n}\sum_{j=1}^n \1{\Phi_j}
   \geq \beta \qquad \text{a.s.}
\end{equation}
Observe that~\eqref{Nk_Normal} clearly implies that
$\IP[\Phi_k]\to\beta$ as $k\to\infty$, but this fact alone is
not enough, since the above events are not independent.
To obtain~\eqref{mnogo_Phi_k}, it is sufficient to prove that 
\begin{equation}
\label{cond_Ak}
 \lim_{k\to\infty}\IP[\Phi_k\mid \D_{k-1}] = \beta \qquad \text{a.s.},
\end{equation}
where~$\D_j$ is the partition generated by the 
events $\Phi_1,\ldots, \Phi_j$. 
In order to prove~\eqref{cond_Ak}, we need to prove (by induction) that
for some $\kappa>0$ we have
\begin{equation}
\label{induct_Ak}
 \kappa\leq \IP[\Phi_k\mid \D_{k-1}] \leq 1-\kappa,
 \quad \text{ for all }k\geq 1.
\end{equation}
Take a small enough $\kappa<\beta$, and let us try to 
do the induction step. Let~$D$ be any event from~$\D_{k-1}$;
\eqref{induct_Ak} implies that $\IP[D]\geq \kappa^{k-1}$.
The following is a standard argument in random interlacements;
see e.g.\ the proof of Lemma~4.5 of~\cite{CT12}
(or Claim~8.1 of~\cite{DRS14}).
% Now, we proceed similarly to (Sznitman, Teixeira, ...).
Abbreviate ${\widehat B}=B_1\cup\ldots \cup B_{k-1}$, and 
let 
\begin{align*}
%  \LL_{11} &= \big\{\text{trajectories that intersect ${\widehat B}$
% and do not intersect $B_k$}\big\},\\
 \LL_{12} &= \big\{\text{trajectories of $\RI(1)$
that first intersect ${\widehat B}$
and then $B_k$}\big\},\\
 \LL_{21} &= \big\{\text{trajectories that first intersect $B_k$
and then ${\widehat B}$}\big\},\\
  \LL_{22} &= \big\{\text{trajectories that intersect $B_k$
and do not intersect ${\widehat B}$}\big\}.
\end{align*}
Also, let $\tL_{12}$ and~$\tL_{21}$ be
independent copies of $\LL_{12}$ and~$\LL_{21}$.
Then, let $N_k^{(ij)}$ and~$\tN_k^{(ij)}$ 
represent the numbers of excursions between~$\partial B_k$
and~$\partial B'_k$ generated 
by the trajectories from~$\LL_{ij}$ and~$\tL_{ij}$
correspondingly.

By construction, we have $N_k=N_k^{(12)}+N_k^{(21)}+N_k^{(22)}$;
also, the random variable 
$\tN_k:=\tN_k^{(12)}+\tN_k^{(21)}+N_k^{(22)}$
is independent of~$D$ and has the same law as~$N_k$.
Observe also that, by our choice of~$b_k$'s, we have
$\ln b_k = \ln^3 b_{k-1}$. 
Define the event
\[
 W_k = \big\{\max\{N_k^{(12)},N_k^{(21)},\tN_k^{(12)},
 \tN_k^{(21)}\} \geq \ln^{17/12} b_k\big\}.
\]
Observe that, by Lemma~\ref{l_escape_from_ball}~(i) and 
Lemma~\ref{l_cap_distantball}~(i), the 
cardinalities of~$\LL_{12}$ and~$\LL_{21}$ have Poisson
distribution with mean~$O(\ln b_{k-1})=O(\ln^{1/3} b_k)$
(for the upper bound, one can use 
that ${\widehat B}\subset \B(2b_{k-1})$).
So, the expected value of all $N$'s in the above display 
is of order $\ln^{1/3}b_k \times \ln b_k = \ln^{16/12}b_k$
(recall that each trajectory generates~$O(\ln b_k)$
excursions between~$\partial B_k$
and~$\partial B'_k$). 
Using a suitable bound on the tails of the 
compound Poisson random variable (see e.g.~(56) of~\cite{CPV15}),
we obtain $\IP[W_k]\leq c_1 \exp(-c_2 \ln^{1/12} b_k)$, so
for any $D\in \D_{k-1}$ (recall that $\ln b_k = e^{3^k}$),
\begin{equation}
\label{small_processes}
 \IP[W_k\mid D\big] \leq \frac{\IP[W_k]}{\IP[D]}
\leq c_1 (1/\kappa)^{k-1}\exp\big(-c_2 e^{3^k/12}\big).
\end{equation}
This implies that (note that 
$\tN_k=N_k-N_k^{(12)}-N_k^{(21)}+\tN_k^{(12)}
 +\tN_k^{(21)}$) 
\begin{align*}
 \IP[\Phi_k\mid D] & = \IP\Big[N_k\leq \frac{2}{\ln\gamma}\ln^2 b_k - 
    q_\beta\frac{\sqrt{6}\ln^{3/2}b_k}{\ln\gamma}
     \;\Big|\; D\Big]\\
&\leq    \IP\Big[W_k^\complement, N_k\leq \frac{2}{\ln\gamma}\ln^2 b_k - 
    q_\beta\frac{\sqrt{6}\ln^{3/2}b_k}{\ln\gamma}
     \;\Big|\; D\Big] + \IP[W_k\mid D\big]\\
 &\leq    \IP\Big[\tN_k\leq \frac{2}{\ln\gamma}\ln^2 b_k - 
    q_\beta\frac{\sqrt{6}\ln^{3/2}b_k}{\ln\gamma}
      + 2\ln^{17/12} b_k\Big]
    + \IP[W_k\mid D\big]\\
    &\to \beta \quad \text{ as }k\to\infty
\end{align*}
since $17/12 < 3/2$ and by~\eqref{small_processes}
(together with an analogous lower bound,
this takes care of the induction step in~\eqref{induct_Ak}
 as well).
So, we have 
\begin{equation}
\label{limsupP[A|D]}
\limsup_{k\to\infty}\IP[\Phi_k\mid \D_{k-1}] \leq
  \beta \qquad \text{a.s.},
\end{equation}
and, analogously, it can be shown that
\begin{equation}
\label{liminfP[A|D]}
\liminf_{k\to\infty}\IP[\Phi_k\mid \D_{k-1}] \geq
  \beta \qquad \text{a.s.},
\end{equation}
We have just proved~\eqref{cond_Ak} and hence~\eqref{mnogo_Phi_k}.

% Now, assume that the excursions between $B_k$ and~$B'_k$, $k\geq 1$,
% were constructed using the soft local times.
% \textbf{(discussion here? define them formally? at least 
% denote them as e.g. $Z^{(k)}_j$\dots)}
Now, let~$(\hZ^{(j),k},j\geq 1)$ be 
the $\RI$'s excursions between $B_k$ and~$B'_k$, $k\geq 1$,
constructed as in Section~\ref{s_SLT}.
Also, for $k\in [\Delta_1,\Delta_2]$ (to be specified later)
let~$(\tZ^{(j),k},j\geq 1)$ be sequences 
of i.i.d.\ excursions, with starting points chosen accordingly 
to~$\hhm_{B_k}^{B'_k}$. We assume that all the above
excursions are constructed simultaneously for
all $k\in [\Delta_1,\Delta_2]$\footnote{we have
chosen to work with finite range of~$k$'s because 
constructing excursions with soft local times on
an infinite collection of disjoint sets requires some additional
formal treatment}. 
Next, let us define the sequence of \emph{independent} events
% (recall the definition of~$\Phi_k$ from~\eqref{df_Phi_k})
\begin{equation}
\label{df_I_k}
 \JJ_k= \Big\{\text{there exists } x\in B_k \text{ such that }
x\notin \tZ^{(j),k} \text{ for all } j\leq
\frac{2}{\ln\gamma}\ln^2 b_k - \ln^{11/9}b_k\Big\},
\end{equation}
that is, $\JJ_k$ is the event that the set~$B_k$
is not completely covered by the first 
$\frac{2}{\ln\gamma}\ln^2 b_k - \ln^{11/9}b_k$
independent excursions.
% Let us fix~$\delta_0>0$ such that 
% $\beta+\frac{\pi}{4\gamma^2}>1+\delta_0$. We are going to
% prove that \textbf{do we need this?}
% \begin{equation}
% \label{mnogo_I_k}
%  \liminf_{n\to\infty} \frac{1}{n}\sum_{j=1}^n \1{\JJ_j}
%    \geq \frac{\pi}{4\gamma^2}-\delta_0 \qquad \text{a.s.}
% \end{equation}

Next, fix~$\delta_0>0$ such that 
$\beta+\frac{\pi}{4\gamma^2}>1+\delta_0$.
Let us prove the following fact:
\begin{lem}
\label{l_compare_RW_RI}
 For all large enough~$k$ it holds that
\begin{equation}
\label{eq_compare_RW_RI}
  \IP[\JJ_k]\geq \frac{\pi}{4\gamma^2}-\delta_0.
\end{equation}
\end{lem}
\begin{proof}
We first outline the proof in the following way: 
\begin{itemize}
 \item consider a simple random walk on a torus
of slightly bigger size (specifically, $(\gamma+\eps_1)b_k^{1/2}$),
so that the set~$B'_k$ would ``completely fit'' there;
 \item we recall a known result that, up to time~$t_k$
(defined just below), the torus is not completely covered
with high probability;
 \item using soft local times, we \emph{couple} 
the i.i.d.\ excursions between~$B_k$ and~$B'_k$ with the 
simple random walk's excursions between the corresponding
sets on the torus (denoted later as~$A$ and~$A'$);
 \item using Lemma~\ref{l_LD_SLT}, we 
conclude in~\eqref{ind_subset_rw} that the set 
of simple random walk's excursions is likely to contain
the set of i.i.d.\ excursions;
 \item finally, we note that the simple random walk's excursions
will not complete cover the set~$A$ with at least constant
probability, and this implies~\eqref{eq_compare_RW_RI}.
\end{itemize}

Note that Theorem~1.2 of~\cite{D12} implies that 
there exists (large enough)~$\hat c$
such that the torus~$\Z^2_m$ is not completely covered
by time $\frac{4}{\pi}m^2\ln^2 m - {\hat c}m^2\ln m \ln\ln m$
with probability converging to~$1$ as~$m\to\infty$. 
Let~$\eps_1$ be a small constant
chosen in such a way that 
$\frac{\pi}{4(\gamma+\eps_1)^2}>\frac{\pi}{4\gamma^2}-\delta_0$. 
Abbreviate
\[
  t_k = \frac{4}{\pi} 
 (\gamma+\eps_1)^2b_k\ln^2 \big((\gamma+\eps_1)b_k^{1/2}\big)
   - {\hat c} (\gamma+\eps_1)^2b_k
\ln \big((\gamma+\eps_1)b_k^{1/2}\big) 
\ln\ln \big((\gamma+\eps_1)b_k^{1/2}\big);
\]
due to the above observation, the
probability that $\Z^2_{(\gamma+\eps_1)b_k^{1/2}}$ is covered
by time~$t_k$ goes to~$0$ as $k\to\infty$.
Let $Z^{(1)},\ldots,Z^{(N^*_{t_k})}$ be the simple random
walk's excursions on the torus $\Z^2_{(\gamma+\eps_1)b_k^{1/2}}$
between~$\partial\B(b_k^{1/2})$ and~$\partial\B(\gamma b_k^{1/2})$.
% where~$\eps_1$ is a small enough constant. 
Assume also that the torus is mapped on~$\Z^2$
in such a way that its image is centered in~$y_k$.
Denote
\[
 m_k =  \frac{2}{\ln\gamma}\ln^2 b_k - (\ln\ln b_k)^2\ln b_k
\]
Then, we
take $\delta = O\big((\ln b_k)^{-1}(\ln\ln b_k)^2\big)$ in 
Lemma~\ref{l_number_exc_torus},
and obtain that
\begin{equation}
\label{N*t_k}
 \IP[N^*_{t_k}\geq m_k] \geq 1-c_1\exp(-c_2(\ln\ln b_k)^4).
\end{equation}
Next, abbreviate (recall~\eqref{df_I_k})
\[
 m'_k = \frac{2}{\ln\gamma}\ln^2 b_k - \ln^{11/9}b_k.
\]

Also, denote $A=\B(b_k^{1/2})$, $A'=\B(\gamma b_k^{1/2})$,
 $A,A'\subset\Z^2_{(\gamma+\eps_1)b_k^{1/2}}$.
Observe that, due to Lemma~\ref{l_compare_RW_RI}
\begin{equation}
\label{hm_hhm}
 \hm_A^{A'}(y) = \hm_{B_k}^{B'_k}(y) = 
\hhm_{B_k}^{B'_k}(y)\big(1+O(b_k^{-1/2})\big).
\end{equation}
We then couple the random walk's excursions~$(Z^{(j)}, j\geq 1)$
with the independent excursions $(\tZ^{(j),k},j\geq 1)$
using the soft local times.
Using Lemma~\ref{l_LD_SLT} (with $\theta=O(\ln^{-8/9}b_k)$)
 and~\eqref{hm_hhm},
we obtain
\begin{equation}
% \begin{align}
% \lefteqn{
 \IP\big[L_{m_k}(y)\geq \hm_{B_k}^{B'_k}(y)
\big(\textstyle\frac{2}{\ln\gamma}\ln^2 b_k - \ln^{10/9}b_k\big)
\text{ for all }y\in\partial B_k\big]
% }\nonumber\\
% & 
\geq 1-c_3 \exp(-c_4 \ln^{2/9} b_k).
% \phantom{********************}
\label{slt_rw>}
% \end{align}
\end{equation}
Let $\ttL_j(y)=({\tilde \xi}_1 + \cdots+{\tilde \xi}_j)
 \hhm_{B_k}^{B'_k}(y)$ be the soft local times for the independent
excursions (as before, ${\tilde \xi}$'s are i.i.d.\ 
Exponential(1) random variables).
Using usual large deviation bounds for sums 
of i.i.d.\ random variables
together with~\eqref{hm_hhm}, we obtain that
\begin{align}
\lefteqn{
 \IP\big[\ttL_{m'_k}(y)\leq \hm_{B_k}^{B'_k}(y)
\big(\textstyle\frac{2}{\ln\gamma}\ln^2 b_k - \ln^{10/9}b_k\big)
\text{ for all }y\in\partial B_k\big]}
\phantom{***************}
\nonumber\\
& = \IP\big[{\tilde \xi}_1 + \cdots+{\tilde \xi}_{m'_k}
  \leq \textstyle\frac{2}{\ln\gamma}\ln^2 b_k - \ln^{10/9}b_k\big]
\nonumber\\
& \geq 1-c_5 \exp(-c_6 \ln^{4/9} b_k).
\label{slt_iid<}
\end{align}
So, \eqref{slt_rw>}--\eqref{slt_iid<} imply that
\begin{equation}
\label{ind_subset_rw}
 \IP\big[\{\tZ^{(j),k},j\leq m'_k\}
  \subset\{Z^{(j)}, j\leq N^*_t\}\big]\geq 
1-c_7 \exp(-c_8 \ln^{2/9} b_k).
\end{equation}

Then, we use the translation invariance of the torus to obtain the 
following:
If $\IP[\Z^2_m \text{ is not completely covered}]\geq c$, and
$A\subset \Z^2_m$ is such that $|A|\geq q m^2$, then
$\IP[A \text{ is not completely covered}]\geq qc$.
So, since 
\[
\big|\B(b_k^{1/2})\big|=\Big(\frac{\pi}{4(\gamma+\eps_1)^2}+o(1)\Big)
\big|\Z^2_{(\gamma+\eps_1)b_k^{1/2}}\big|, 
\]
 Lemma~\ref{l_compare_RW_RI} now follows from~\eqref{N*t_k}
and~\eqref{ind_subset_rw}.
\end{proof}

Now, abbreviate (recall~\eqref{df_Phi_k} and~\eqref{mnogo_Phi_k})
\[
 m''_k = \frac{2}{\ln\gamma}\ln^2 b_k - 
    q_\beta\frac{\sqrt{6}\ln^{3/2}b_k}{\ln\gamma},
\]
and, being $\hL^{(k)}$ the soft local time 
of the excursions of random interlacements
between~$\partial B_k$ and~$\partial B'_k$,
 define the events
\begin{equation}
\label{df_Mk}
 M_k = \big\{ \hL^{(k)}_{m''_k}(y) 
  \leq \ttL^{(k)}_{m'_k}(y)
\text{ for all }y\in\partial B_k\big\}.
\end{equation}
Note that on~$M_k$ it holds that
$\{\hZ^{(j),k},j\leq m''_k\} \subset \{\tZ^{(j),k},j\leq m'_k\}$.

Then, we need to prove that
\begin{equation}
\label{compare_RI_indep_excrus}
  \IP[M_k]\geq 1 - c_9\ln^2 b_k \exp(-c_{10}\ln^{2/3}b_k).
\end{equation}
Indeed, first, analogously to~\eqref{slt_iid<}
we obtain (note that $\frac{11}{9}<\frac{4}{3}<\frac{3}{2}$)
\begin{equation}
\label{slt_iid>}
  \IP\big[\ttL_{m'_k}(y)\geq \hm_{B_k}^{B'_k}(y)
\big(\textstyle\frac{2}{\ln\gamma}\ln^2 b_k - \ln^{4/3}b_k\big)
\text{ for all }y\in\partial B_k\big]
 \geq 1-c_{11} \exp(-c_{12} \ln^{2/3} b_k).
\end{equation}
Then, we use Lemma~\ref{l_LD_SLT} with $\theta=O(\ln^{-1/2}b_k)$ to obtain that
\begin{equation}
\label{slt_ri<}
  \IP\big[\hL_{m''_k}(y)\leq \hm_{B_k}^{B'_k}(y)
\big(\textstyle\frac{2}{\ln\gamma}\ln^2 b_k - \ln^{4/3}b_k\big)
\text{ for all }y\in\partial B_k\big]
 \geq 1-c_{13} \exp(-c_{14} \ln b_k),
\end{equation}
and \eqref{slt_iid>}--\eqref{slt_ri<}
 imply~\eqref{compare_RI_indep_excrus}.

Now, it remains to observe that on the event $\Phi_k\cap\JJ_k\cap M_k$
the set~$B_k$ contains at least one vacant site.
By~\eqref{mnogo_Phi_k}, \eqref{eq_compare_RW_RI},  
and~\eqref{compare_RI_indep_excrus}, 
one can choose large enough $\Delta_1<\Delta_2$ such that,
with probability arbitrarily close to~$1$,
there is $k_0\in[\Delta_1,\Delta_2]$ such that
 $\Phi_{k_0}\cap\JJ_{k_0}\cap M_{k_0}$ occurs.
This concludes the proof of~Theorem~\ref{t_critical}.
\end{proof}

\section*{Summary of notation}
For reader's convenience, we include here a brief summary 
of notation used in this paper:

\begin{itemize}
 \item $\B(y,r)$: the ball centered in~$y$ and of radius~$r$,
with respect to the Euclidean norm;
 \item $S=(S_n, n\geq 0)$: the two-dimensional simple random walk;
 \item $\tau_0(A)$ and $\tau_1(A)$: entrance and hitting times
of set~$A$, cf.~\eqref{entrance_t} and~\eqref{hitting_t};
 \item $a(\cdot)$: the potential kernel of the 
two-dimensional simple random walk, cf.~\eqref{def_a(x)};
 \item $\hm_A(\cdot)$ and $\capa(A)$: harmonic measure 
and capacity of set~$A$ with respect to the simple
random walk in dimension $d\geq 2$ 
(see~\eqref{def_hm} and \eqref{df_cap2} for $d=2$ and
the beginning of Section~\ref{s_toolbox} for higher dimensions);
 \item $\s=(\s_n, n\geq 0)$: Doob's $h$-transform of~$S$,
with respect to~$a$ (informally, two-dimensional simple random walk
conditioned on not hitting the origin);
 \item $\htau_0(A)$ and $\htau_1(A)$: entrance and hitting times
of set~$A$ of the walk~$\s$;
 \item $\he_A(\cdot)$ and $\hhm_A(\cdot)$: equilibrium
and harmonic measures on~$A$ with respect to~$\s$, see formulas below \eqref{df_cap2};
 \item $\V^\alpha$: the vacant set of two-dimensional
random interlacements on level~$\alpha$;
 \item $H_A(x,y)$: the Poisson kernel of simple
random walk, cf.~\eqref{df_Poisson_kernel};
 \item $G(x,y)$ and $G_A(x,y)$: the Green's function
of the simple random walk in $d\geq 3$, and the 
Green's function restricted on a finite set~$A$, $d\geq 2$;
 \item $\Es_A(x)$: escape probability from~$A$, starting
at~$x$ (see the beginning of Section~\ref{s_toolbox});
 \item $\Z^2_n$: the two-dimensional torus, 
$\Z^2_n=\Z^2/n\Z^2$;
 \item $X=(X_k, k\geq 0)$: simple random walk on the 
two-dimensional torus;
 \item $L_k(\cdot)$ and $\hL_k(\cdot)$: soft local times
(at $k$th excursion) of the excursion processes with respect 
to simple random walk on the torus and random interlacements,
cf.~\eqref{df_SLT_torus} and~\eqref{df_SLT_RI};
 \item $\hm_A^{A'}(\cdot)$ and $\hhm_A^{A'}(\cdot)$: 
harmonic measures on~$A\subset A'$ with respect to~$A'$,
for simple random walk on the torus and random interlacements,
cf.~\eqref{df_hm_set_RW} and~\eqref{df_hm_set_RI};
 \item $\psi(\cdot,\cdot)$ and $\hpsi(\cdot,\cdot)$:
invariant measures for the process of first/last sites
of excursions, 
cf.~\eqref{df_psi} and~\eqref{df_hpsi};
 \item $\tS$: conditioned random walk on the annulus
 $\B((1+\eps)n)\setminus \B(n)$, cf.~\eqref{def_tilde_P};
 \item $\tH_A(u,v)$: the Poisson kernel with respect to~$\tS$,
 cf.~\eqref{tilde_H}.
\end{itemize}

\section*{Acknowledgments}
We thank Greg Lawler for very valuable advice on the proof 
of Proposition~\ref{p_SRW_Hoelder}. We also thank 
Diego de Bernardini, Christophe Gallesco, and the referee for careful
reading of the manuscript and valuable comments and suggestions. 
This work was partially supported by CNPq
and MATH-AmSud project LSBS.

\end{document}